\documentclass[12pt, reqno]{amsart}
\usepackage{geometry}
\usepackage{amssymb,xspace,mathrsfs,lmodern,xcolor}

\usepackage[shortlabels]{enumitem}
    \setenumerate[0]{label={\rm (\roman*)}, leftmargin=*}

\usepackage[utf8]{inputenc}
\usepackage{hyperref}

\usepackage[textsize=footnotesize,textwidth=15ex]{todonotes}

\usepackage[capitalize]{cleveref}

\numberwithin{equation}{section}

\newtheorem{thm}{Theorem}[section]
\newtheorem{lem}[thm]{Lemma}
\newtheorem{prop}[thm]{Proposition}
\newtheorem{cor}[thm]{Corollary}

\newtheorem{thmintro}{Theorem}

\newtheorem{corintro}[thmintro]{Corollary}

\theoremstyle{definition}
\newtheorem{remark}[thm]{Remark}

\newtheorem{definition}[thm]{Definition}

\DeclareMathOperator{\Aut}{Aut}
\DeclareMathOperator{\Out}{Out}
\DeclareMathOperator{\Sym}{Sym}
\DeclareMathOperator{\Stab}{Stab}
\DeclareMathOperator{\Fix}{Fix}

\DeclareMathOperator{\dist}{dist}
\DeclareMathOperator{\Isom}{Isom}

\newcommand{\bigast}{\mathop{\scalebox{1.5}{\raisebox{-0.2ex}{$\ast$}}}}
\newcommand{\dash}{\nobreakdash-\hspace{0pt}}

\DeclareMathOperator{\Ch}{Ch}

\DeclareMathOperator{\proj}{proj}

\begin{document}

\title[Lattice envelopes of RAAGs]{Lattice envelopes of right-angled Artin groups}

\author{Pierre-Emmanuel \textsc{Caprace}}
\address{Pierre-Emmanuel Caprace, UCLouvain,
Institut de Recherche en Mathématiques et Physique (IRMP),
Chemin du Cyclotron 2, box L7.01.02,
1348 Louvain-la-Neuve, Belgium}
\email{pe.caprace@uclouvain.be}

\author{Tom  \textsc{De Medts}}
\address{Tom De Medts, Ghent University, Department of Mathematics, Computer Science and Statistics, Krijgslaan 281, S9, 9000 Gent, Belgium.}
\email{tom.demedts@ugent.be}

\thanks{}

\begin{abstract}
Let $\Gamma$ be a finite simplicial graph with at least two vertices, and let $G(\Gamma)$ be the associated right-angled Artin group. We describe a locally compact group $\mathcal U$ containing $G(\Gamma)$ as a cocompact lattice. If $\Gamma$ is not a join (i.e., the complement graph is connected), then the group $\mathcal U$  is non-discrete, almost simple, but not virtually simple: it has a smallest normal subgroup $\mathcal U^+$ which is an open simple subgroup, and the quotient $\mathcal U/\mathcal U^+$ is isomorphic to the right-angled Coxeter group $W(\Gamma)$. 
Under suitable assumptions on $\Gamma$, we rely on work by Bader--Furman--Sauer and Huang--Kleiner to show that $\mathcal U \rtimes \Aut(\Gamma)$ is the universal lattice envelope of $G(\Gamma)$: for every lattice envelope $H$ of $G(\Gamma)$, there is a continuous proper homomorphism $H \to \mathcal U \rtimes \Aut(\Gamma)$.  In particular, no lattice envelope of $G(\Gamma)$ is virtually simple. We also show that no locally compact group quasi-isometric to $G(\Gamma)$ is virtually simple. This contrasts with the case of free groups. The group $\mathcal U$ is a universal automorphism group of the Davis building of $G(\Gamma)$, with prescribed local actions.  As an application, we describe the algebraic structure of the full automorphism group of the Cayley graph of $G(\Gamma)$ with respect to its standard generating set. 
\end{abstract}

\maketitle


\begin{flushright}
\begin{minipage}[t]{0.60\linewidth}\itshape\small
Catégorique\\
angle droit du caractère\\
de l’esprit du coeur\\
Je me suis miré dans ce caractère\\
et m’y suis trouvé\\
trouvé chez moi
\vspace{2mm}

\hfill\upshape \textemdash Le Corbusier,  \emph{Le poème de l'angle droit}, 1955.
\end{minipage}
\end{flushright}

\section{Introduction}

Following H.~Furstenberg~\cite{Fur67}, we define a  \textbf{lattice envelope} of a countable discrete group $\Lambda$ as a locally compact group $H$ containing $\Lambda$ as a lattice. The problem of determining all possible lattice envelopes of a given group $\Lambda$, up to a natural equivalence relation, is proposed  and thoroughly examined by Bader--Furman--Sauer in their recent work \cite{BFS}. In this paper, we address that question in the case where $\Lambda$ is a right-angled Artin group (abbreviated \textbf{RAAG} in the following). 

Recall that the \textbf{right-angled Artin group} associated with a graph $\Gamma$ is a group defined by generators and relations, having a generator  $x_s$ for each vertex $s$ of $\Gamma$, and the relator $[x_s, x_t]= x_s^{-1} x_t^{-1} x_s x_t $ for each pair of vertices $\{s, t\}$ forming an edge of $\Gamma$. The class of RAAGs is of fundamental importance in geometric group theory and its applications, notably  through the notion of special cube complexes introduced by Haglund--Wise in \cite{HagWis}. 

It is well-known that the right-angled Artin group $G(\Gamma)$ acts regularly on the chambers of a right-angled building $\mathcal B$ of type $\Gamma$, see \cite[Corollary~11.7]{Davis}. It follows that the group $G(\Gamma)$ may be interpreted as a \textit{universal group} of automorphisms of $\mathcal B$ with \textit{prescribed local actions}. The notion of \textbf{universal groups with prescribed local actions} was first introduced by Burger--Mozes \cite{BuMo1} for groups acting locally finite trees, later extended to arbitrary trees by Smith~\cite{Smith}, and then from trees to right-angled buildings by De Medts--Silva--Struyve~\cite{DMSS} and Bossaert--De Medts~\cite{BDM21}. Fixing for each vertex $s \in S = V(\Gamma)$ an infinite cyclic group $X_s$ with a regular permutation action on an infinite set~$\Omega_s$, and setting $\mathbf X = (X_s)_{s \in S}$, the  associated universal group 
$\mathcal U(\mathbf X) \leq \Aut(\mathcal B)$ is naturally isomorphic to $G(\Gamma)$; see Proposition~\ref{prod:graph-products-are-univ-groups}.  Now,  for each vertex $s$, let $Y_s \leq \Sym(\Omega_s)$ be the unique subgroup   isomorphic to the infinite dihedral group~$D_\infty$, and containing $X_s$ as a subgroup of index~$2$. Set $\mathbf Y = (Y_s)_{s \in S}$. Then we have  $G(\Gamma) \cong \mathcal U(\mathbf X) \leq \mathcal U(\mathbf Y)$. It turns out that  the latter inclusion is a cocompact lattice envelope: $\mathcal U(\mathbf Y)$ is a totally disconnected locally compact group containing $\mathcal U(\mathbf X)$ as a cocompact lattice; see Proposition~\ref{prop:LattEnv}. Moreover, the group $\mathcal U(\mathbf Y)$ is non-discrete as soon as $\Gamma$ is not the complete graph (i.e., $G(\Gamma)$ is not abelian). 

The goal of this paper is twofold. The first one is to elucidate the algebraic structure of the lattice envelope $\mathcal U(\mathbf Y)$ of $G(\Gamma)$. The second is to determine to what extent $\mathcal U(\mathbf Y)$ is the largest lattice envelope of $G(\Gamma)$. 

The first goal is achieved in Section~\ref{sec:Univ-Groups}, where general results on the structure of universal groups are obtained, see Theorems~\ref{thm:simplicity} and~\ref{thm:U(F/F+)} (we refer to \cite[Theorem~4.12]{RST} for a closely related result in the context of groups acting on trees). We record the following statement, which is a direct consequence of those:

\begin{thmintro}\label{thmintro:1}
Suppose that the graph $\Gamma$ is irreducible, with at least two vertices. Then the intersection of all non-trivial normal subgroups of  
$G = \mathcal U(\mathbf Y)$ coincides with  the subgroup $G^+$ generated by the stabilizers of the chambers. Moreover, $G^+$ is simple, and  the quotient $G/G^+$ is isomorphic to the Coxeter group $W(\Gamma)$. 
\end{thmintro}

The graph $\Gamma$ is called \textbf{irreducible} if it is not a join, i.e. the complement graph is connected. That condition is equivalent to the condition that the right-angled Artin group $G(\Gamma)$ (or equivalently the Coxeter group $W(\Gamma)$) does not split as a direct product of two non-trivial subgroups. 

Theorem~\ref{thmintro:1} implies that the lattice envelope  $\mathcal U(\mathbf Y)$ of $G(\Gamma)$ is almost simple, but not virtually simple since $W(\Gamma)$ is infinite.

To tackle the second goal, our starting point consists in combining  the work of Bader--Furman--Sauer~\cite{BFS} with   known properties of RAAGs. This  allows us to   show that if $\Gamma$ is connected with at least two vertices, then every lattice envelope $H$ of a finite index subgroup of $G(\Gamma)$ is a cocompact envelope; see Proposition~\ref{prop:envelope->coco:RAAG} (we also refer to \cite[Theorem~10.1]{HH23} for a related statement). It follows that the locally compact group $H$ is quasi-isometric to~$G(\Gamma)$. Quasi-isometric rigidity of RAAGs has been studied extensively, see notably \cite{BKS,Hu17,HK18}. We shall invoke 
the results by Huang--Kleiner \cite{HK18}, which apply under the hypotheses  that $G(\Gamma)$ is not cyclic and that $\Out(G(\Gamma))$ is finite.  The first condition just means that $\Gamma$ contains more than one vertex. The second condition  can be characterized by two simple conditions on the graph $\Gamma$. In order to formulate those, we continue to denote the vertex set of $\Gamma$   by $S$. Moreover, for each $s \in S$, we denote by  $s^\perp$ the collection of  those vertices $t \neq s$ that are adjacent to $s$, and we set $s \cup s^\perp = \{s\} \cup s^\perp$. 
Then $\Out(G(\Gamma))$ is finite if and only if the following two conditions hold (see \cite{Laurence}):

\begin{itemize}
\item[(R1)] For each $s \in S$, the induced graph on $S \setminus (s \cup s^\perp)$ is connected. 

\item[(R2)] For all $s, t \in S$, if $s^\perp \subseteq t  \cup t^\perp$, then $s=t$. 
\end{itemize}

In \cite{CF12,Day12}, condition (R1) is called \textbf{star $2$-connectedness}. A pair $s,t \in S$ with $s^\perp \subseteq t \cup t^\perp$ as in condition (R2) but with $s \neq t$ is called a \textbf{domination pair} in \cite{Day12}.

It is easy to see that a graph satisfying (R1) and (R2) must be connected (see Lemma~\ref{lem:connected} below).	
We will also consider the following extra condition on $\Gamma$:
\begin{itemize}
\item[(R3)] For each vertex $s \in S$, the pointwise stabilizer of $ s \cup s^\perp$ in $\Aut(\Gamma)$ is trivial.
\end{itemize}

The condition (R3)  appears in various earlier references, notably \cite{HagPau_simpl},  \cite{Cla11}, \cite{Hu18} and  \cite{BF23}. A graph $\Gamma$ satisfying (R3) is called \textbf{star-rigid} in \cite{Hu18}.  Explicit examples of graphs satisfying (R1), (R2) and (R3) are provided by $n$-cycles (with $n \geq  5$), by the $1$-skeleta of the cube, the dodecahedron and the icoasahedron, or by the $1$-skeleta of their barycentric subdivisions.
In fact, a ``generic'' finite graph (with respect to the probability model of selecting each edge with probability $\tfrac{1}{2}$) will satisfy (R1), (R2) and (R3). Indeed, \cite[Theorem~5.1]{CF12} shows that generically, $\Out(G(\Gamma))$ is finite and hence (R1) and (R2) hold; on the other hand, almost all finite graphs have trivial automorphism group \cite{ER63}, so also (R3) holds generically. Moreover, the complement graph must also satisfy (R1) and (R2) for the same reasons, hence it is connected (see Lemma~\ref{lem:connected} below). This shows that, asymptotically almost surely, a random finite graph $\Gamma$ satisfies (R1), (R2), (R3) and is irreducible. 	(The statement in \cite[Theorem~5.1]{CF12} allows for a more general probability parameter, and has been further sharpened in \cite{Day12}.)

We shall prove the following (see Theorem~\ref{thm:Coxeter-action} below).

\begin{thmintro}\label{thmintro:2}
Let $\Gamma$ be a finite graph satisfying  (R1),  (R2), (R3), and let    $H$ be a compactly generated locally compact group quasi-isometric to $G(\Gamma)$.

Then there is a continuous homomorphism $H \to W(\Gamma) \rtimes \Aut(\Gamma)$ whose image is of finite index. In particular, $H$ has a discrete quotient commensurable with~$W(\Gamma)$.
%
%
\end{thmintro}

In the special case where $H$ is finitely generated  and where $\Gamma$ satisfies  (R1), (R2) and (R3) and does not contain any induced subgraph isomorphic to a $4$-cycle, it follows from \cite[Theorem~1.2]{Hu18} that $H$ is abstractly commensurable to $G(\Gamma)$. In particular, $H$ has a quotient  commensurable with $W(\Gamma)$ because $G(\Gamma)$ does. It is thus important to underline that Theorem~\ref{thmintro:2} applies also  in the case where  $\Gamma$ does contain induced $4$-cycles. 

Combining Theorem~\ref{thmintro:2} with the aforementioned facts on lattice envelopes, we obtain the following (see Theorem~\ref{thm:Main-latt-env}). 

\begin{thmintro}	\label{thmintro:3}			
Let $\Gamma$ be a finite  graph 
satisfying  (R1), (R2) and (R3), and let $H$ be a locally compact group containing $G(\Gamma)$ as a lattice. Then there is a continuous, proper homomorphism $\alpha \colon H \to \mathcal U(\mathbf Y) \rtimes \Aut(\Gamma)$, whose restriction to $G(\Gamma)$ is an isomorphism of $G(\Gamma)$ to $\mathcal U(\mathbf X)$, where the notation is as in Theorem~\ref{thmintro:1}.  
\end{thmintro}

It follows that the Coxeter group quotient appearing in Theorem~\ref{thmintro:1} is not a specific feature of the lattice envelope considered there (see Corollary~\ref{cor:discrete-quot} below):

\begin{corintro}\label{corintro:Coxeter-quot}
Let $\Gamma$ and $H$ be as in Theorem~\ref{thmintro:3}. Then $H$ has an infinite discrete quotient containing a finite index subgroup isomorphic to $W(\Gamma)$. 
\end{corintro}

Theorem~\ref{thmintro:2} can also be combined with results by Horbez--Huang on measure equivalence of RAAGs, thereby showing that the Coxeter group quotient from Theorem~\ref{thmintro:2} also occurs for groups that are integrably measure equivalent  to $G(\Gamma)$,  see \cite[Theorem~1]{HH23}. 

We remark that, if $\Gamma$ has at least two vertices but  no edge (so that (R2) fails), then $G(\Gamma)$ is a non-abelian free group. Its lattice envelopes are described by \cite[Theorem~C]{BFS}; the standard lattice envelopes are virtually simple (see \cite{Tits}). This is in sharp contrast with Corollary~\ref{corintro:Coxeter-quot}, which shows that if $\Gamma$ satisfies the conditions of Theorem~\ref{thmintro:3}, then no lattice envelope of $G(\Gamma)$ is virtually simple. 

Our considerations remain valid for the class of all graph products of $2$-ended groups. This class  is strictly larger than the class of RAAGs. Any such graph product over $\Gamma$ is abstractly commensurable to $G(\Gamma)$ (see \cite[Theorem~1]{JS01} or Proposition~\ref{prop:commensurable} below). This allows us to derive the following result (see Corollary~\ref{cor:different-envelopes} below), that may be viewed as a generalization of \cite[Theorem~1.9]{HK18}:

\begin{corintro}\label{corintro:different-envelopes}
Let $\Gamma$ be a finite irreducible graph with at least two vertices satisfying (R1), (R2) and (R3). Then there exist an infinite family  of groups $(G_n)$, that are all commensurable with $G(\Gamma)$, such that for all $m \neq n$, the groups $G_m$ and $G_n$ do not have a common lattice envelope. 
\end{corintro}

Related results applying to graphs $\Gamma$ that are reducible have recently been established by Horbez--Huang, see \cite[Theorems~4 and 11.8]{HH23}.

A very natural lattice envelope of $G(\Gamma)$ is the full automorphism group of its Cayley graph $X(\Gamma)^{(1)}$ with respect to the standard generating set. The choice of notation is motivated by the fact that the Cayley graph in question coincides with the $1$-skeleton of the (special) CAT(0) cube complex $X(\Gamma)$ defined as the universal cover of the Salvetti complex associated with $\Gamma$. It turns out that the universal group
$\mathcal U(\mathbf Y)$ from Theorem~\ref{thmintro:1} naturally acts on $X(\Gamma)^{(1)}$ by automorphisms. Using that observation, we shall derive  the following. 

\begin{corintro}\label{corintro:Salvetti}
Let $\Gamma$ be a finite graph and let $H = \Aut(X(\Gamma)^{(1)})$ be endowed with the compact-open topology. 
\begin{enumerate}
\item $H$ is discrete if and only if $\Gamma$ is the complete graph. 
\end{enumerate}
Suppose in addition that $\Gamma$ is irreducible, 
and satisfies  (R1), (R2) and (R3). Then the following assertions hold. 
\begin{enumerate}[resume]
\item  $H \cong \mathcal U(\mathbf Y) \rtimes \Aut(\Gamma)$.
\item $H$ is almost simple, its smallest normal subgroup is the subgroup $H^+$ generated by the vertex stabilizers and preserving the natural edge-colouring of $X(\Gamma)^{(1)}$ by the elements of $S$, and the quotient $H/H^+$ is isomorphic to $W(\Gamma) \rtimes \Aut(\Gamma)$. 
\item Every lattice envelope $L$ of $G(\Gamma)$ has a compact normal subgroup $K$ such that $L/K$ is isomorphic to a closed subgroup of $H$ containing $G(\Gamma)$.
\end{enumerate}

\end{corintro}

The first item in Corollary~\ref{corintro:Salvetti} recovers a result established by T.~Taylor \cite{Tay17}, that has also been recovered recently by Hartnick--Incerti-Medici in \cite[Theorem~E]{HIM}. We refer to \cite[Corollary~3.18]{MSSW23} for a closely related result. We emphasize that the assertions (ii), (iii), (iv) do not hold for arbitrary graphs $\Gamma$. Indeed, consider for example the graph $\Gamma$ with two vertices and no edge; property (R2) does not hold for this graph. The group $G(\Gamma)$ is a free group of rank~$2$, while $W(\Gamma)$ is infinite dihedral. In that case, the graph $X(\Gamma)^{(1)}$ is the regular tree of degree~$4$. Hence, the group $ \Aut(X(\Gamma)^{(1)})$ has a simple subgroup of index~$2$ (see \cite{Tits}), hence it does not have any non-trivial action on the Cayley  graph of $W(\Gamma)$, which is the simplicial line. 

\subsection*{Acknowledgements}
We thank Thomas Weigel for asking  a question that gave us the impetus to undertake this project. It was started while we were visiting the University of Münster in September 2023 on the occasion of the Focus Programme
entitled \textit{Actions of totally disconnected locally compact groups on discrete structures}. We thank that institution for its hospitality and support, and the organizers of that Programme for their invitation. We are also grateful to Jingyin Huang,  Merlin Incerti-Medici and the referee for their comments on a first draft of this paper. 

Both authors are supported in part by the FWO and the F.R.S.-FNRS under the EOS programme (project ID 40007542). 

\section{Automorphisms of the chamber-graph of a building}\label{sec:chamber-graph}

Let $(W, S)$ be a Coxeter system. The \textbf{Davis diagram} of $(W, S)$ is the edge-labeled graph $\Gamma = \Gamma(W, S)$ with vertex set $S$, where two distinct vertices $s, t \in S$ form an edge  if the Coxeter number $m_{st} = o(st)$ is finite; in that case we set $n = m_{st}$. In case $(W, S)$ is right-angled, we omit the labels and $\Gamma$ is then a simple graph. The usual Coxeter diagram of $(W, S)$ is denoted by $\mathrm{Diag}(W, S)$. Notice that both diagrams carry the exact same information. In particular, we have $\Aut(\mathrm{Diag}(W, S)) = \Aut(\Gamma(W, S))$. We identify that group with the group of \textbf{diagram automorphisms} of $W$, i.e. the subgroup of $\Aut(W)$ that stabilizes the Coxeter generating set $S$.

Let now $\mathcal B$ be a building of type $(W, S)$. We mainly view $\mathcal B$ as a chamber system: the elements of $\mathcal B$ are chambers, and for each $s \in S$, the set $\mathcal B$ is equipped with an equivalence relation denoted $\sim_s$ and called \textbf{$s$-adjacency}.
A \textbf{gallery} is a finite sequence $c_0 \sim_{s_1} c_1 \sim_{s_2} \cdots \sim_{s_n} c_n$ of adjacent chambers. It is called a \textbf{minimal gallery} if there is no shorter gallery from $c_0$ to $c_n$.
Moreover, $\mathcal B$ is endowed with a \textbf{Weyl distance} function $\delta \colon \mathcal B\times \mathcal B \to W$ with the property that if 
$$c_0 \sim_{s_1} c_1 \sim_{s_2} \cdots \sim_{s_n} c_n$$
 is a gallery joining the chamber  $c_0$ to $c_n$ in $\mathcal B$, then there is a subword $(s_{i_1}, \dots, s_{i_m})$ of $(s_1, \dots, s_n)$ such that  $\delta(c_0, c_n) = s_{i_1} \dotsm s_{i_m} \in W$,  see \cite[Lemma~5.16]{AB08} (by a \textbf{subword}, we mean that the sequence  $(s_{i_1}, \dots, s_{i_m})$ is obtained from $(s_1, \dots, s_n)$ by deleting a certain number of terms). An \textbf{automorphism} of the building $\mathcal B$ is a permutation of the chambers that preserves the Weyl distance. We denote by $\Aut(\mathcal B)$ the group formed by those automorphisms. They are sometimes called the \textbf{type-preserving automorphisms}. The group $\Aut(\mathcal B)$ has a natural extension formed by the so-called \textbf{type-permuting automorphisms}, defined as those permutations $g \in \Sym(\mathcal B)$ for which there is a diagram automorphism $\pi(g)  \in \Aut(\Gamma(W, S))$ such that for all $x, y \in \mathcal B$, we have
$$\delta(g.x, g.y) = \pi(g).\delta(x, y).$$
We denote by $\widetilde{\Aut(\mathcal B)}$ the group of all type-permuting automorphisms. 
Observe that the map $\pi \colon  \widetilde{\Aut(\mathcal B)} \to \Aut(\Gamma(W, S))$ is a homomorphism whose kernel is $\Aut(\mathcal B)$.

It is well known (see \cite[Proposition~A.14]{AB08}) that the full automorphism group of the standard simplicial realization of a building $\mathcal B$ coincides with $\widetilde{\Aut(\mathcal B)}$. This is no longer the case for the Davis realization of a building, as is easily seen by considering the Coxeter group of rank~$3$ defined as the free product of $3$ copies of the cyclic group of order~$2$. Nonetheless, in some special cases, there is enough rigidity to ensure that the automorphism group of   a building coincides with the automorphism group of its Davis realization.

The \textbf{chamber-graph} $\mathscr C(\mathcal B)$ of a building $\mathcal B$ is defined as the graph whose vertex set is the set of chambers, where two vertices form an edge if the corresponding chambers are adjacent. In the special case of thin buildings, the following result recovers the easier part of  \cite[Theorem~5.12]{HagPau_simpl} and \cite[Corollary~B]{BF23}.

\begin{prop}\label{prop:types}
Let $(W, S)$ be a Coxeter system. 
Suppose that  for each $s \in S$, the only automorphism of $\Gamma(W, S)$ fixing pointwise the $1$-ball around $s$ is the trivial one.

Then the natural injective homomorphism
$\widetilde{\Aut(\mathcal B)} \to \Aut(\mathscr C(\mathcal B))$
is surjective. Hence  we have $\widetilde{\Aut(\mathcal B)} \cong \Aut(\mathscr C(\mathcal B))$.
%
\end{prop}

We will use the following subsidiary facts. 

\begin{lem}\label{lem:chamber-graph-1}
The panels in $\mathcal B$ are the maximal cliques in $\mathscr C(\mathcal B)$. 
\end{lem}
\begin{proof}
If three pairwise distinct vertices $x, y, z$ in $\mathscr C(\mathcal B)$ are mutually adjacent, then the gallery $(x, y, z)$ joining $x$ to $z$ via $y$ is not minimal. It follows that $x, y, z$ belong to the same panel.
\end{proof}

A subgraph of a graph is called  \textbf{convex} if every minimal path joining two vertices in that subgraph is entirely contained in the subgraph. A \textbf{minimal circuit} is an induced subgraph isomorphic to a circuit of length $k$, and whose convex hull is of girth $k$.

\begin{lem}\label{lem:chamber-graph-2}
Let $x, y, z \in \mathcal B$ be chambers and $s, t \in S$ be distinct elements such that $x \sim_s y \sim_t z$. 

Then $m_{st} < \infty$ if and only if the path $(x, y, z)$ is contained in a minimal circuit of length $2m_{st}$ in $\mathscr C(\mathcal B)$. 
\end{lem}
\begin{proof}
The `only if' part is clear, since residues in buildings are convex (see \cite[Example 5.44(b)]{AB08}). For the converse, let $y'$ be the chamber belonging to the given circuit, at distance $m_{st}$ from $y$. Then that circuit gives rise to two distinct minimal paths joining $y$ to $y'$, one passing through $x$ and the other through $z$.
Now let $\mathcal R$ be the rank~$2$ residue containing $x,y,z$, which is of type $\{ s,t \}$, and consider the projection $\proj_\mathcal R$ onto that residue (see \cite[Definition 5.35]{AB08}).
The projection of the minimal path $(y, x, \dots, y')$ gives a (possibly stammering) path $(y, x, \dots, \proj_\mathcal R(y'))$ in $\mathcal R$, while
the projection of the minimal path $(y, z, \dots, y')$ gives a (possibly stammering) path $(y, z, \dots, \proj_\mathcal R(y'))$ in $\mathcal R$.
Joining these two paths results in a circuit containing $x, y, z$ and $\proj_\mathcal R(y')$ in $\mathcal R$ of length $\leq 2 m_{st}$, with equality if and only if the given circuit is already entirely contained in $\mathcal R$. By minimality of the given circuit, the conclusion follows.
\end{proof}

\begin{lem}\label{lem:chamber-graph-3}
Let $x\in \mathcal B$ be a chamber and $s, t \in S$ be distinct elements such that  $m_{st} < \infty$. 

Then the $\{s, t\}$-residue containing $x$ coincides with the union of all minimal circuits of length $2m_{st}$ in $\mathscr C(\mathcal B)$ containing $x$, a chamber $s$-adjacent to $x$ and a chamber $t$-adjacent to $x$. 
\end{lem}
\begin{proof}
The  $\{s, t\}$-residue containing $x$  consists of chambers located on such circuits. Conversely, given such a circuit, the same argument as in the previous lemma shows that it is entirely contained in the residue in question. 
\end{proof}

\begin{proof}[Proof of Proposition~\ref{prop:types}]
Let $g \in \Aut(\mathscr C(\mathcal B))$.  

By Lemma~\ref{lem:chamber-graph-1}, the automorphism $g$ permutes the panels. In particular, for each chamber $c  \in \mathcal B$,    there is a permutation $\tau_c$ of $S$ such that $g$ maps the $s$-panel containing $c$ to the $\tau_c(s)$-panel containing $g(c)$. 

In view of Lemma~\ref{lem:chamber-graph-2}, we have $m_{st} = m_{\tau_c(s) \tau_c(t)}$ for all $s, t \in S$. Thus $\tau_c \in \Aut(\Gamma(W, S))$. 

We must prove that $\tau_c$ is independent of $c$. It suffices to show that $\tau_c = \tau_d$ for any two adjacent chambers $c, d$. Suppose that $c$ and $d$ are $s$-adjacent for some $s  \in S$. Hence $\tau_c(s) = \tau_d(s)$. Let $t \in S$ be different from $s$  and adjacent to $s$ in $\Gamma(W, S)$. Hence $m_{st} < \infty$.
  It follows from Lemmas~\ref{lem:chamber-graph-1}, \ref{lem:chamber-graph-2} and \ref{lem:chamber-graph-3} that $g$ permutes the spherical residues of rank~$2$. Thus $g$ maps the  $\{s, t\}$-residue containing $c$ to the $\{\tau_c(s), \tau_c(t)\}$-residue containing $g(c)$. Since $d$ is $s$-adjacent to $c$, it belongs to the former residue, and since $\tau_c(s) = \tau_d(s)$, we infer that $\tau_c(t) = \tau_d(t)$. It follows that $\tau_c$ and $\tau_d$ coincide on the $1$-ball around $s$ in the Davis diagram $\Gamma(W, S)$. In view of the hypothesis made on $\Gamma(W, S)$, it follows that $\tau_c = \tau_d$. 
  
The foregoing discussion yields a map $ \Aut(\mathscr C(\mathcal B)) \to  \Aut(\Gamma(W, S))$, which is clearly a group homomorphism. The kernel consists of those permutations of the chambers that permute the panels in a type-preserving way. Those are precisely the automorphisms of $\mathcal B$.
\end{proof}

\section{Automorphism groups of right-angled buildings with prescribed local action}\label{sec:Univ-Groups}
\subsection{Semi-regular right-angled buildings}\label{ss:RABs}

Throughout we fix a finite graph $\Gamma$, and consider the associated  right-angled Artin group $G(\Gamma)$, and right-angled Coxeter group $W(\Gamma)$. 

Let $S$ denote the vertex set of $\Gamma$. For $s, t \in S$, we  write $s \sim t$ if $s$ and $t$  are adjacent vertices of  $\Gamma$. 
For each $s \in S$, we write
\[ s^\perp := \{ t \in S \mid s \sim t, \ s\neq t\}. \]
We also set $s \cup s^\perp = \{s \} \cup s^\perp \subseteq S$. 
Now let $\mathcal B$ be a right-angled building of type~$\Gamma$.
An \textbf{$s$-tree-wall} is a residue of type $s \cup s^\perp$; these tree-walls play an important role in the structure of right-angled buildings.

A right-angled building $\mathcal B$ of type $\Gamma$ is called \textbf{semi-regular} if for each $s \in S$, any two $s$-panels have the same cardinality, say $q_s \geq 2$. In that case, we say that $\mathcal B$ is semi-regular of thickness $\mathbf q = (q_s)_{s \in S}$. 
 We recall from \cite[Proposition~1.2]{HagPau} that up to isomorphism, there is a unique semi-regular right-angled building  of type $\Gamma$ and of thickness $\mathbf q$. 
Recall from \cite[Corollary~11.7]{Davis} that this building can be constructed as follows. 

For each $s \in S$, let us fix a set $\Omega_s$ of cardinality $q_s$.
Choose a group $X_s$ which is in one-to-one correspondence with $\Omega_s$. Let $\mathbf X = (X_s)_{s \in S}$ and 
$$G_\Gamma(\mathbf X) = \bigast_{s \in S} X_s \big\slash \langle \! \langle [X_s, X_t]  : s, t \in S \text{ with } \{s, t\} \in E(\Gamma)\rangle \! \rangle$$ 
be the graph product of those groups $(X_s)_{s \in S}$ associated with the graph $\Gamma$. Then the chamber graph $\mathscr C(\mathcal B)$ of the building $\mathcal B$ can be viewed as the Cayley graph of $G_\Gamma(\mathbf X)$ with respect to the generating set $\bigcup_{s \in S} X_s$. The $s$-panels are the left cosets of $X_s$ in $G_\Gamma(\mathbf X)$. 

We assume now that if $s, t \in S$ belong to the same $\Aut(\Gamma)$-orbit, then $q_s = q_t$. In other words the map  $s \mapsto q_s$ is constant on the $\Aut(\Gamma)$-orbits. We assume in addition that for such a pair $s, t$, the groups $X_s$ and $X_t$ are two copies of the same group. Then the automorphism group   $\Aut(\Gamma)$ has a natural action  on $G_\Gamma(\mathbf X)$ by automorphisms, that  preserves the generating set $\bigcup_{s \in S} X_s$ and permutes the subsets $(X_s)_{s \in S}$ according to the $\Aut(\Gamma)$ action on the vertex set $S$. Thus it induces an action of $\Aut(\Gamma)$ on the set of chambers of $\mathcal B$ that preserves the adjacency, but permutes the set types. We may thus view $G_\Gamma(\mathbf X)$,   $\Aut(\mathcal B)$ and $ \Aut(\Gamma) = \Aut(W(\Gamma), S)$ as subgroups of  the type-permuting automorphism group $\widetilde{\Aut(\mathcal B)}$ (see \S\ref{sec:chamber-graph}). Since the natural homomorphism $\widetilde{\Aut(\mathcal B)} \to \Aut(\Gamma)$ is surjective in this case, 
we obtain the following. 

\begin{lem}\label{lem:autom-chamber-graph}
Let $\mathcal B$ be a semi-regular right-angled building of type $\Gamma$ such that the map $s \mapsto q_s$ is constant on the $\Aut(\Gamma)$-orbits. If $\Gamma$ satisfies (R3), then we have
$$ \Aut(\mathscr C(\mathcal B)) \cong \widetilde{\Aut(\mathcal B)} \cong \Aut(\mathcal B) \rtimes \Aut(\Gamma).$$
\end{lem}
\begin{proof}
Condition (R3) ensures that the type of $\mathcal B$ satisfies the hypothesis of Proposition~\ref{prop:types}. Hence the conclusion follows from the discussion preceding the proof. 
\end{proof}

\subsection{Legal colourings and universal groups}

Let $\mathcal B$ be a semi-regular right-angled building of type $\Gamma$ and thickness $\mathbf q = (q_s)_{s \in S}$. For each $s \in S$, let again $\Omega_s$ denote a set  of cardinality $q_s$. A \textbf{legal colouring} of $\mathcal B$ is a map 
$$\lambda \colon \mathcal B \to \prod_{s \in S} \Omega_s : c \mapsto (\lambda_s(c))_{s \in S}$$
such that for every $s \in S$, the map $\lambda_s \colon \mathcal B \to \Omega_s$ has the following properties:
\begin{enumerate}
\item the restriction of $\lambda_s$ to each $s$-panel $\mathcal P$ is a bijection of $\mathcal P$ to $\Omega_s$,
\item the restriction of $\lambda_s$ to each $t$-panel , with $t \neq s$, is constant.
\end{enumerate}

Let also $X_s \leq \Sym(\Omega_s)$ be a group acting regularly on $\Omega_s$, let $\mathbf X = (X_s)_{s \in S}$ and let $G_\Gamma(\mathbf X)$ be the graph product of those groups over $\Gamma$. In view of the defining relations of $G_\Gamma(\mathbf X)$, we see that for each $s \in S$, the identity on $X_s$ extends to a surjective homomorphism 
$$\rho_s \colon G_\Gamma(\mathbf X) \to X_s$$
whose restriction on $X_t$ is trivial for all $t \neq s$. As recalled above, the group $G_\Gamma(\mathbf X)$ acts regularly on the chambers of $\mathcal B$. Fix a base chamber $c_0 \in \mathcal B$. For each $c \in \mathcal B$, there is a unique $g_c \in G_\Gamma(\mathbf X)$ mapping $c_0$ to $c$. 

\begin{lem}\label{lem:legal-colouring}
The map $\lambda \colon \mathcal B \to \prod_{s \in S} X_s : c \mapsto \big(\rho_s(g_c)\big)_{s \in S}$ is a legal colouring. 
\end{lem}
\begin{proof}
For each $t \in S$, the $t$-panel containing $c_0$ coincides with the orbit $X_t(c_0)$. Therefore, for each chamber $c$, the $t$-panel containing $c$ is the set 
$$\mathcal P = g_c X_t(c_0) = \{g_c x(c_0) \mid x \in X_t\}.$$ 
In other words, $\mathcal P$ corresponds to the left $X_t$-coset of $G_\Gamma(\mathbf X)$ containing $g_c$. 
In view of the definition of $\rho_s$, it directly follows that if $s = t$, the restriction of $\rho_s$ to $\mathcal P$ is a bijection onto $X_s$, whereas if $s \neq t$, then the restriction of $\rho_s$ to $\mathcal P$ is constant. 
\end{proof}

Given a legal colouring $\lambda \colon \mathcal B \to \prod_{s \in S} \Omega_s$,  and an automorphism $g \in \Aut(\mathcal B)$, the \textbf{local action} of $g$ at some $s$-panel $\mathcal P$ is defined as 
$$\sigma_\lambda(g, \mathcal P) = \lambda_s|_{g \mathcal P} \circ g \circ (\lambda_s|_{\mathcal P})^{-1}.$$
This is an element of the permutation group $\Sym(\Omega_s)$.  

Let us now fix a permutation group $F_s \leq \Sym(\Omega_s)$ for each $s \in S$. We set $\mathbf F = (F_s)_{s \in S}$.  Following \cite{DMSS, BDM21}, we  define
\[ \mathcal U(\mathbf F) = \mathcal U^\lambda_{\mathcal B}(\mathbf F) \leq \Aut(\mathcal B) \]
as the set of all those $g \in \Aut(\mathcal B)$ whose local action at every $s$-panel belongs to $F_s$, for all $s \in S$. It is the called the 
 \textbf{universal group with  prescribed local action}. We refer to \cite[Definition 2.14]{BDM21} for detailed information and background. Its  conjugacy class in $\Aut(\mathcal B)$ does not depend on the choice of the legal colouring $\lambda$, see \cite[Proposition~2.15]{BDM21}. 
 Examples of universal groups are provided by the following. 
 
\begin{prop}\label{prod:graph-products-are-univ-groups}
Let  $X_s \leq \Sym(\Omega_s)$ be a group acting regularly on $\Omega_s$ and set  $\mathbf X = (X_s)_{s \in S}$. We identify $X_s$ with $\Omega_s$ via an orbit map. We view the graph product  
$G_\Gamma(\mathbf X)$  as a subgroup of $\Aut(\mathcal B)$, and let $\lambda$ be the legal colouring   of $\mathcal B$ afforded by Lemma~\ref{lem:legal-colouring}.
\begin{enumerate}
\item  $G_\Gamma(\mathbf X)$  is isomorphic to the universal group  $\mathcal U^\lambda_{\mathcal B}(\mathbf X)$, where $X_s$ is identified with a subgroup $\Sym(X_s)$ via the regular action of $X_s$ on itself by left-multiplication.  
\item Let $H \leq \Aut(\mathcal B)$ be a group containing $G_\Gamma(\mathbf X)$. Let $c_0 \in \Ch(\mathcal B)$ and for each $s \in S$, let $\mathcal P_s$ be the $s$-panel containing $c_0$. Let also $F_s \leq \Sym(X_s)$ be the permutation group of $X_s$  formed by the local action at $\mathcal P_s$ of all elements of $\Stab_H(\mathcal P_s)$. Then $H$ is contained in the universal group $\mathcal U^\lambda_{\mathcal B}(\mathbf F)$.
\end{enumerate}
\end{prop}
\begin{proof}
(i) We retain the notation from Lemma~\ref{lem:legal-colouring}. For each $s \in S$ we identify $X_s$ with a subgroup of $\Sym(X_s)$.
Let $c \in \mathcal B$ and $s \in S$. The $s$-panel $\mathcal P$ containing $c$ coincides with $g_c X_s(c_0)$. Given $g \in G_\Gamma(\mathbf X)$, the panel $g \mathcal P$ is thus equal to $g g_c X_s(c_0)$. Therefore, the  local action of $g$ at $\mathcal P$ is the permutation 
$$X_s \to X_s \colon x \mapsto \rho_s(g g_c x) =\rho_s(g g_c) x,$$
which is the left multiplication by a fixed element of $X_s$. Hence it is contained in $X_s$. This proves that $G_\Gamma(\mathbf X)  \leq \mathcal U^\lambda_{\mathcal B}(\mathbf X)$. 

Since each $X_s \leq \Sym(X_s)$ acts freely, it follows from \cite[Proposition~3.14]{BDM21} that the universal group $\mathcal U^\lambda_{\mathcal B}(\mathbf X)$ acts freely on $\mathcal B$. On the other hand, $G_\Gamma(\mathbf X)$ acts transitively on $\mathcal B$. Hence we have $G_\Gamma(\mathbf X)  = \mathcal U^\lambda_{\mathcal B}(\mathbf X)$. 

\medskip \noindent (ii)
Let $\mathcal P$ be an $s$-panel for some $s \in S$, and let $g \in H$. 
In order to compute the local action of $g$ at $\mathcal P$, we choose elements $\tau, \tau' \in G_\Gamma(\mathbf X)$ such that $\tau(c_0) \in \mathcal P$ and $\tau'(c_0) \in g \mathcal P$. In particular, $\tau(\mathcal P_s) = \mathcal P$ and $\tau'(\mathcal P_s) = g\mathcal P$, so that 
$(\tau')^{-1} g \tau \in \Stab_H(\mathcal P_s)$. We set 
$$\varphi_g = \sigma_\lambda\big((\tau')^{-1} g \tau , \mathcal P_s\big).$$
Hence we have 
$$(\tau')^{-1} g \tau x(c_0) = \varphi_g(x)(c_0)$$
for all $x \in X_s$. Therefore $ g \tau x(c_0) = \tau' \varphi_g(x)(c_0)$. By definition of the legal colouring $\lambda$, we also have 
$$(\lambda|_{\mathcal P})^{-1} \colon X_s \to \mathcal P : x \mapsto \tau \rho_s(\tau)^{-1}x(c_0),$$
so that 
$$\sigma_\lambda(g, \mathcal P) : x \mapsto \rho_s\bigg(\tau' \varphi_g\big(\rho_s(\tau)^{-1}x\big)\bigg) = \rho_s(\tau') \varphi_g\big(\rho_s(\tau)^{-1}x\big),$$
since $\varphi_g\big(\rho_s(\tau)^{-1}x\big) \in X_s$. By hypothesis $X_s \leq G_\Gamma(\mathbf X) \leq H$. Moreover, $X_s$ stabilizes $\mathcal P_s$, and for each $y \in X_s$, its local action on $\mathcal P_s$ is the regular permutation $ X_s \to X_s: x \mapsto yx$. This ensures that 
$$\sigma_\lambda(g, \mathcal P) = \sigma_\lambda\big( \rho_s(\tau') , \mathcal P_s\big) \sigma_\lambda\big((\tau')^{-1} g \tau , \mathcal P_s\big)\sigma_\lambda\big(\rho_s(\tau)^{-1}  , \mathcal P_s\big) $$ 
is the composite of three permutations in $F_s$, thereby confirming   that $H \leq \mathcal U^\lambda_{\mathcal B}(\mathbf F)$. 
\end{proof}
		 
\begin{remark}\label{rem:transitive}
We emphasize that Proposition~\ref{prod:graph-products-are-univ-groups}(ii) is no longer true if one replaces the hypothesis that $G_\Gamma(\mathbf X) \leq H$ by the weaker condition that $H$ is transitive on $\mathcal B$. Indeed, consider for example the group $H = \Sym(n)$. Let $C_n \leq H$ be a cyclic subgroup of order~$n$ acting regularly on $\{1, \dots, n\}$. Hence  we have a \textbf{strict factorization} $H = C_n \Sym(n-1)$, which means that every element of $H$ can be written in a unique way as a product of an element of $C_n$ with an element of $\Sym(n-1)$. The group $H$ acts regularly on the chamber of a right-angled building $\mathcal B$ of spherical type and rank~$2$, i.e., a \textbf{generalized $2$-gon}, with type set $S = \{s, t\}$. The chamber set is identified with $H$, while the panels of type $s$ (resp. $t$) are the left cosets of $C_n$ (resp. $\Sym(n-1)$). As  $H = C_n \Sym(n-1)$ is a strict factorization,   the chamber graph  of $\mathcal B$ is a complete bipartite graph, so that $\mathcal B$ is indeed a generalized $2$-gon. The local action $F_s$ (resp. $F_t$) of the stabilizer of an $s$-panel (resp. $t$-panel) is isomorphic to $C_n$ (resp. $\Sym(n-1)$). The universal group $\mathcal U(F_s, F_t)$ is isomorphic to the direct product $C_n \times \Sym(n-1)$. Obviously $H = \Sym(n)$ is not isomorphic to a subgroup of $\mathcal U(F_s, F_t) \cong C_n \times \Sym(n-1)$, so the conclusion of Proposition~\ref{prod:graph-products-are-univ-groups}(ii)  does not hold in this case.
\end{remark}

\subsection{Almost simplicity}

A group $G$ is called \textbf{monolithic} if the intersection of all normal subgroups other than $\{e\}$ is non-trivial. That intersection is then called the \textbf{monolith} of $G$. We say that $G$ is \textbf{almost simple} if it is monolithic, and if its monolith is a non-abelian simple group. 

The  goal of this section is to provide conditions ensuring that $\mathcal U(\mathbf F)$ is almost simple, and to describe its monolith. 
To this end, we shall consider the following normal subgroups of $\mathcal U(\mathbf F)$:
\begin{itemize}
\item $\mathcal U(\mathbf F)^+$ is the subgroup generated by all chamber stabilizers;
\item $\mathcal U(\mathbf F)^\dagger$ is the subgroup generated by the pointwise stabilizers of all maximal spherical residues;
\item $\mathcal U(\mathbf F)^\ddagger$ is the subgroup generated by the pointwise stabilizers of all tree-walls.
\end{itemize}
Clearly, we have inclusions $\mathcal U(\mathbf F)^\ddagger \leq \mathcal U(\mathbf F)^\dagger \leq \mathcal U(\mathbf F)^+$. 
The relevance of the subgroup $\mathcal U(\mathbf F)^\ddagger$ comes from the following result, established in \cite{BDM23}.

\begin{prop}\label{prop:simplicity}
Assume that $\Gamma$ is irreducible with more than one vertex,   and that not all permutation groups $F_s \leq \Sym(\Omega_s)$ act freely on $\Omega_s$.
Then $\mathcal U(\mathbf F)$ is almost simple, and its monolith is  $\mathcal U(\mathbf F)^\ddagger$. 
\end{prop}
\begin{proof}
This follows from \cite[Proposition~7.22]{BDM23} if $\mathcal B$ is locally finite.
The general case, the proof of which is completely similar, can be found in \cite[Proposition~3.4.8]{JensPhD}. (Notice that the irreducibility assumption from \cite[Proposition~3.4.7]{JensPhD} is required for the proof of his Proposition 3.4.8.)
\end{proof}
	
It is thus natural to take a closer look at the group 	$\mathcal U(\mathbf F)^\ddagger$. We record the following observation, which holds without any assumption on the type or on the local actions. Under suitable assumptions, this will allow us to deduce that $\mathcal U(\mathbf F)^\ddagger$ is an open subgroup of $\mathcal U(\mathbf F)$. 
	
\begin{prop}\label{prop:ddagger}
We have $\mathcal U(\mathbf F)^\dagger = \mathcal U(\mathbf F)^\ddagger$.
\end{prop}
\begin{proof}
We work by induction on the rank $r$, defined as the cardinality of the vertex set $S= V(\Gamma)$. If $r=0$ or $r=1$, there is nothing to prove. 

Let $G = \mathcal U(\mathbf F)$ and set $G^\ddagger =  \mathcal U(\mathbf F)^\ddagger$.
Let also $\mathcal R \subseteq \mathcal B$ be a maximal spherical residue. We must show that its pointwise stabilizer $\Fix_G(\mathcal R)$ is contained in $G^\ddagger$. 

Choose $s \in S$ an element belonging to the type of $\mathcal R$, and let $\mathcal T$ be a residue of type $s \cup s^\perp$ containing a chamber of $\mathcal R$. Thus $\mathcal T$ is an $s$-tree-wall and we have $\mathcal R \subseteq \mathcal T$. 

Observe that  the pointwise stabilizer $\Fix_G(\mathcal T)$ of $\mathcal T$ is contained in $G^\ddagger$ by definition. Therefore, in order to show that $\Fix_G(\mathcal R) \leq G^\ddagger$, it suffices to show that the subgroups $\Fix_G(\mathcal R)$ and $\Fix_G(\mathcal R) \cap G^\ddagger$, which are both contained in $\Stab_G(\mathcal T)$, have the same image under the canonical projection $\Stab_G(\mathcal T) \to \Stab_G(\mathcal T)/\Fix_G(\mathcal T)$. We consider the latter quotient as a subgroup of the automorphism group $\Aut(\mathcal T)$ of the residue $\mathcal T$, viewed as a right-angled building.

Choose an $s$-panel $\mathcal T_s$ in $\mathcal R \cap \mathcal T$. We may decompose $\mathcal T$ as the product $\mathcal T_s \times \mathcal T_0$, where  $ \mathcal T_0$ is a residue of type $s^\perp$ containing a chamber of $\mathcal T_s$. Accordingly, we have a decomposition $\mathcal R \cong \mathcal T_s \times \mathcal R_0$. The image of $\Fix_G(\mathcal R)$  in $\Aut(\mathcal T)$ acts trivially on the factor $\mathcal T_s$ of this decomposition. Let us look at the image in  $\Aut(\mathcal T_0)$. 

The residue $\mathcal T_0$ is a right-angled building of type $J$, where $J \subset S$ has cardinality at most~$\leq r-1$. 
Moreover, we know by \cite[Lemma~3.14]{DMSS} that the image of $\Stab_G(\mathcal T_0)$ in  $\Aut(\mathcal T_0)$ is itself a universal group of the form $H = \mathcal U_{\mathcal T_0}(\mathbf F|_J)$, where the local action $\mathbf F|_J$ is the tuple $(F_t)_{t \in J}$. 
The induction hypothesis applies to $H$, and ensures that $H^\dagger \leq H^\ddagger$. In particular, we have 
$$\Fix_H(\mathcal R_0) \leq H^\ddagger.$$
We claim that $H^\ddagger$ coincides with the image of $\Stab_G(\mathcal T_0) \cap G^\ddagger $ in $\Aut(\mathcal T_0)$. Indeed,   every tree-wall $\mathcal W$ of the residue $\mathcal T_0$ can be viewed as the intersection with  $\mathcal T_0$ of a tree-wall $\widetilde{\mathcal W}$ of $\mathcal B$, whose type contains $s$. In addition, every element of $\Fix_{H}(\mathcal W)$ can be extended to an element of $G$ fixing $\widetilde{\mathcal W}$ pointwise, as a consequence of \cite[Proposition~3.8]{BDM21}.  This shows that the image of $\Fix_{G}(\widetilde{\mathcal W})$ in $\Aut(\mathcal T_0)$ coincides with $\Fix_{H}(\mathcal W)$. Since this holds for every tree-wall $\mathcal W \subseteq \mathcal T_0$, the claim follows.

By definition, the image of $\Fix_G(\mathcal R_0)$ in $\Aut(\mathcal T_0)$ is contained in $\Fix_H(\mathcal R_0)$. We have seen above that  $\Fix_H(\mathcal R_0) \leq H^\ddagger$. Therefore,  the claim implies that  $\Fix_G(\mathcal R_0)$ and $\Fix_G(\mathcal R_0) \cap G^\ddagger$ have the same image in $\Aut(\mathcal T_0)$. 
 
In view of the canonical decompositions $\mathcal T \cong \mathcal T_s \times \mathcal T_0$ and  $\mathcal R \cong \mathcal T_s \times \mathcal R_0$, and since $\Fix_G(\mathcal R)$ acts trivially on the factor $\mathcal T_s$, we see that the natural image of $\Fix_G(\mathcal R)$ in $ \Aut(\mathcal T) \cong \Aut(\mathcal T_s) \times \Aut(\mathcal T_0)$ coincides with the image of $\Fix_G(\mathcal R) \cap G^\ddagger$. This finishes the proof.
\end{proof}
	
We shall now show that the equality $\mathcal U(\mathbf F)^+ = \mathcal U(\mathbf F)^\dagger = \mathcal U(\mathbf F)^\ddagger$ holds if the type is irreducible. For that purpose, we need to consider again the projection map between residues in buildings. In the chamber system approach to building, we recall that given two residues $\mathcal R_1, \mathcal R_2$ in  a building $\mathcal B$, the projection $\proj_{\mathcal R_1}(\mathcal R_2)$ is a residue whose chamber set coincides with $\{\proj_{\mathcal R_1}(c) \mid c \in \Ch(\mathcal R_2)\}$, where $\proj_{\mathcal R_1}(c)$ is the unique chamber of $\mathcal R_1$ at minimal numerical distance from $c$, see \cite[Definition~5.35 and Lemma~5.36]{AB08}. 

\begin{lem}\label{lem:irred-Coxeter}
Let $\Gamma$ be irreducible and let $\mathcal A$ be a  building of right-angled type~$\Gamma$. 
For each spherical residue  $\mathcal R$ in $\mathcal A$   and each chamber $c \in \mathcal R$, there is a  tree-wall $\mathcal T$  such that $\proj_{\mathcal R}(\mathcal T) = \{c\}$.
\end{lem}
\begin{proof}
It  follows from from \cite[Lemma~5.36 and Proposition~5.37]{AB08} that for any choice of  chambers $(c_1, c_2) \in \Ch(\mathcal R_1)\times \Ch(\mathcal R_2)$ and for any apartment $\mathcal A$ containing $c_1$ and $c_2$, the residue $\proj_{\mathcal R_1}(\mathcal R_2)$ contains a chamber of $\mathcal A$, and the type of $\proj_{\mathcal R_1}(\mathcal R_2)$  is entirely determined by considering the intersections $\mathcal A\cap \mathcal R_1$ and $\mathcal A\cap \mathcal R_2$, viewed as residues in the thin building $\mathcal A$. (For a description of $\proj_{\mathcal R_1}(\mathcal R_2)$ in the simplicial complex approach to buildings instead of the chamber system approach, and an alternative argument showing   that it can be computed within $\mathcal A$, we  refer to \cite[Proposition~5.37]{AB08} or to \cite[3.19]{Tits74}.)

Therefore, in order to prove the lemma, it suffices to consider the special case where the given building $\mathcal A$ is thin. We assume henceforth that it is the case. 
Since the type is right-angled, it follows that the Davis realization $|\mathcal A|$ is a CAT(0) cube complex whose $1$-skeleton coincides with the Cayley graph of $W(\Gamma)$ with respect to its natural Coxeter generating set. The hyperplanes of $|\mathcal A|$ coincide with the walls of $\mathcal A$. Moreover, an $s$-tree-wall is nothing but the union of all $s$-panels across a given wall. 

We claim that   $|\mathcal A|$ contains a pair of  $H_1, H_2$ of  \textbf{strongly separated} hyperplanes:  this means that no other hyperplane of $|\mathcal A|$ crosses both $H_1$ and $H_2$. 
In order to see that, we shall invoke \cite[Proposition~5.1]{CapSag}. This requires checking that the hypotheses of the latter are satisfied, namely  the following three conditions: the cube complex $|\mathcal A|$ is irreducible, the $W(\Gamma)$-action has no fixed point at infinity, and the $W(\Gamma)$-action on $|\mathcal A|$ is \textbf{essential}, i.e.,  the $W(\Gamma)$-orbits contain points arbitrarily far away from any half-space. 

The irreducibility of $|\mathcal A|$ means that the set of hyperplanes cannot be partitionned into two non-empty subsets such that each hyperplane in the first crosses each hyperplane in the second. This would give rise to a partition of the set of reflections in $W(\Gamma)$ into two commuting subsets, contradicting that $W(\Gamma)$ is of irreducible type. 

If $\xi$ were a point at infinity fixed by all elements of $W(\Gamma)$, it would be fixed by each reflection. Let $\mathcal R_0$ be a maximal spherical residue of $\mathcal A$. It corresponds to a maximal cube $|\mathcal R_0|$ in $|\mathcal A|$. Let $p$ be the center of $|\mathcal R|$. 
Each reflection $r$ preserving $\mathcal R_0$ fixes pointwise the geodesic ray $[p, \xi)$. Hence that ray  is entirely contained in the hyperplane  $H$ associated with $r$. Since $|\mathcal R_0|$ is a maximal cube, the intersection of all those hyperplanes is $\{p\}$. Thus $W(\Gamma)$ cannot fix $\xi$. 

Let $H$ be a hyperplane in $|\mathcal A|$. Let $r$ be the reflection fixing $H$ pointwise. There is $w \in W(\Gamma)$ such that $wrw^{-1}$ is a Coxeter generator $s \in S$. Since $\Gamma$ is irreducible, there exists $t \in S$ such that $m_{st} = \infty$. It follows that $r' = w^{-1} t w$ is a reflection fixing pointwise a hyperplane $  H'$ that does not cross $  H$. Therefore the orbits of the subgroup $\langle r, r'\rangle$ contain points that are arbitrarily far away from each of the two half-spaces defined by $H$. It follows that the  $W(\Gamma)$-action on $|\mathcal A|$ is {essential}. 

This shows that the hypotheses of  \cite[Proposition~5.1]{CapSag} are satisfied, hence there is a pair of strongly separated hyperplanes $H_1, H_2$. 
Let $D \leq W(\Gamma)$ be the dihedral group generated by the reflections through $H_1$ and $H_2$. Any two hyperplanes in the $D$-orbit of $H_1$ are strongly separated. We can thus find some $H_0$ in that orbit, such that $H_0$ does not cross any  hyperplane of the cube  $|\mathcal R|$ associated with the spherical residue  $\mathcal R$ given by hypothesis. 

Let $\mathcal T_0$ be the tree-wall in $\mathcal A$ corresponding to $H_0$. It follows that $\proj_{\mathcal R}(\mathcal T_0)$ is reduced to a single chamber, say $c_0$.  The stabilizer $\Stab_{W(\Gamma)}(\mathcal R)$ contains a unique element $w$ sending $c_0$ to $c$. It follows that the tree-wall $\mathcal T = w(\mathcal T_0)$ has the required property. 
\end{proof}

\begin{prop}\label{prop:+-equal-ddagger}
If $\Gamma$ is irreducible, then we have $\mathcal U(\mathbf F)^+ =  \mathcal U(\mathbf F)^\ddagger$. 
\end{prop}
\begin{proof}
Let $G = \mathcal U(\mathbf F)$. We have $G^\ddagger \leq G^\dagger \leq G^+$. Hence, it suffices to show that $G^+ \leq G^\ddagger$. 
In other words, given a chamber $c$ in $\mathcal B$, we must show that $G_c \leq G^\ddagger$. 

Let  $\mathcal R$ be a maximal spherical residue containing $c$. 
By Lemma~\ref{lem:irred-Coxeter}, there exists a tree-wall $\mathcal T$ such that $\proj_{\mathcal R}(\mathcal T) = \{c\}$. By definition we have $\Fix_G(\mathcal T) \leq G^\ddagger$. 

Let $J$ be the type of $\mathcal R$. For each $j \in J$, let $\sigma_j$ be the $j$-panel containing $c$. Fix $g \in G_c$. For each $j$, we invoke \cite[Proposition~3.8]{BDM21}. This  affords an element $g_j \in G$ such that $g|_{\sigma_j} = (g_j)|_{\sigma_j} $, which moreover fixes all chambers $d$ with $\proj_{\sigma_j}(d) = c$. In particular, $g_j$ fixes pointwise $\mathcal T$, hence $g_j \in G^\ddagger$. Moreover, $g_j$ acts trivially on $\sigma_k$ for all $k \in J \setminus\{j\}$. In particular, the product $\prod_{j \in J} g_j$, taken in an arbitrary order, is an element of $G^\ddagger$ that stabilizes $\sigma_j$ and acts on $\sigma_j$ like $g|_{\sigma_j}$, for all $j \in J$. Since $\mathcal R$ has a canonical decomposition as the direct product $\prod_{j \in J} \sigma_j$, the restriction $g|_{\mathcal R}$  is determined by the tuple $(g|_{\sigma_j})_{j\in J}$. We infer that the elements $g \in G_c $ and $\prod_{j \in J} g_j \in G_c \cap \Fix_G(\mathcal T)$ have the same action on the set of chambers of  $\mathcal R$. 

This shows that the groups  $G_c $ and $G_c \cap \Fix_G(\mathcal T)$, which are both contained in $\Stab_G(\mathcal R)$,  have the same image under the canonical projection $\Stab_G(\mathcal R) \to \Stab_G(\mathcal R)/\Fix_G(\mathcal R)$. Since $\Fix_G(\mathcal R) \leq G^\ddagger$ by  Proposition~\ref{prop:ddagger},  it follows that $G_c \leq G^\ddagger$, as required. 
\end{proof}

The following theorem is now immediate. 

\begin{thm}\label{thm:simplicity}
Assume that $\Gamma$ is irreducible with more than one vertex,   and that not all permutation groups $F_s \leq \Sym(\Omega_s)$ act freely on $\Omega_s$. 
Then the group $\mathcal U(\mathbf F)^+$ is simple. Moreover, every non-trivial normal subgroup of $\mathcal U(\mathbf F)$ contains $\mathcal U(\mathbf F)^+$. 
\end{thm}
\begin{proof}
In view of Proposition~\ref{prop:+-equal-ddagger}, this is a reformulation of Proposition~\ref{prop:simplicity}.
\end{proof}
	
\begin{remark}
Theorem~\ref{thm:simplicity} can be rephrased by saying that the group $\mathcal U(\mathbf F)$ is monolithic, with monolith equal to $\mathcal U(\mathbf F)^+$.  Moreover, if $\mathcal U(\mathbf F)$ is endowed with the permutation topology associated with its action on the set of chambers, it follows that the monolith $\mathcal U(\mathbf F)^+$ is open. In other words every non-trivial normal subgroup of $\mathcal U(\mathbf F)$ is open. This strengthens \cite[Proposition~7.23]{BDM23}. 
\end{remark}

\subsection{\texorpdfstring{On the quotient $\mathcal U(\mathbf F)/\mathcal U(\mathbf F)^+$}{On the quotient U(F)/U(F)+}}

Our next goal is to describe the quotient group $\mathcal U(\mathbf F)/\mathcal U(\mathbf F)^+$. This requires to introduce additional notation. For each $s \in S$, we denote by $F_s^+$ the subgroup of the permutation group $F_s \leq \Sym(X_s)$ generated by all point stabilizers. It is obviously a normal subgroup. We view the quotient group $F_s/F_s^+$ as a permutation group on the set $X_s^0 = F_s^+\backslash X_s$ of $F_s^+$-orbits. 
    Let $S' := \{ s \in S \mid |X_s^0| > 1 \}$, let $\Gamma'$ be the Coxeter subsystem of $\Gamma$ induced by $S'$, and let $\mathcal B'$ be the semi-regular building of type $\Gamma'$ with color sets $X_s^0$ for each $s \in S'$.
We define $\mathcal U(\mathbf F/\mathbf F^+)$ as the universal group of automorphisms of $\mathcal B'$ with local action prescribed by $\mathbf F/\mathbf F^+ = (F_s/F_s^+)_{s \in S'}$. 

The main result of this section asserts that the groups $\mathcal U(\mathbf F)/\mathcal U(\mathbf F)^+$ and $\mathcal U(\mathbf F/\mathbf F^+)$ are isomorphic, see~\cref{thm:U(F/F+)} below. 
The proof relies on  the concept of \emph{imploding} a right-angled building, introduced in  \cite[\S 5]{BDM21}, that we now recall.

\begin{definition}
    Let $\mathcal B$ be a semi-regular right-angled building of type $\Gamma$, with color sets $X_s$ for each $s \in S$. Assume that for each $s \in S$, we have an equivalence relation $\equiv_s$ on $X_s$ and let $Y_s$ be the corresponding set of equivalence classes; for each $x \in X_s$, we write $[x] \in Y_s$ for the corresponding equivalence class.
    Let $S' := \{ s \in S \mid |Y_s| > 1 \}$, let $\Gamma'$ be the Coxeter subsystem of $\Gamma$ induced by $S'$, and let $\mathcal B'$ be the semi-regular building of type $\Gamma'$ with color sets $Y_s$ for each $s \in S'$ and fix a legal coloring $\lambda' \colon \mathcal B' \to \prod_{s \in S'} Y_s$.
    
    Choose a base chamber $c_0$ in $\mathcal B$ and a base chamber $c_0'$ in $\mathcal B'$ such that  $\lambda'_s(c_0') = [\lambda_s(c_0)]$ for all $s$.
    By \cite[Proposition 5.1]{BDM21}, there is a non-expansive surjective map $\tau \colon \mathcal B \to \mathcal B'$ that is color-preserving, in the sense that $\lambda'_s(\tau(c)) = [\lambda_s(c)]$ for all chambers $c \in \mathcal B$ and all $s \in S'$. Moreover, by \cite[Remark 5.4]{BDM21}, the map $\tau$ is the unique such map if one imposes in addition that $\tau(c_0) = c'_0$. 
    The map $\tau$ is called an \textbf{implosion} of $\mathcal B$ with respect to the given equivalence relations. That construction will  be generalized in Proposition~\ref{prop:implosion} below.
\end{definition}

We will apply this specifically to the situation described above, where the equivalence classes on $X_s$ are the orbits of $F_s^+$, for each $s \in S$. Notice that $|Y_s| = 1$ precisely if $F_s^+$ acts transitively on $X_s$. We continue to denote the imploded building by $\mathcal B^0$.

\begin{prop}\label{prop:G+-orbits}
    Let $c,d \in \mathcal B$ be two chambers and set $G =\mathcal U(\mathbf F)$. Then $\tau(c) = \tau(d)$ if and only if $c$ and $d$ are in the same $G^+$-orbit.
\end{prop}
\begin{proof}
    Assume first that $c$ and $d$ are in the same $G^+$-orbit.
    As in the proof of \cite[Proposition 5.3]{BDM21}, we may assume that $d = g.c$ for some $g \in G$ stabilizing some chamber $c_0$ and proceed by induction on $\dist(c_0, c)$, the result being obvious when $c = c_0$.
    (Notice that in \cite[Proposition 5.3]{BDM21}, the imploded building is given by the orbits of $F_s$ rather than $F_s^+$, so we cannot directly quote that result.)
    
    So assume that $\dist(c_0, c) = n$ and let $e$ be a chamber $s$-adjacent to $c$ at distance $n-1$ from $c_0$, for some $s \in S$. By the induction hypothesis, we have $\tau(e) = \tau(g.e)$.

    If the $s$-colors of $c$ and $e$ lie in the same $F_s^+$-orbit, then the same is true for $g.c$ and $g.e$, since  the local action of $g$ on each $s$-panel is contained in $F_s^+$    by \cite[Lemma 3.16]{BDM21}. 
    By the definition of $\tau$, we then have $\tau(c) = \tau(e)$ and $\tau(g.c) = \tau(g.e)$.
    Combined with $\tau(e) = \tau(g.e)$, we conclude that $\tau(c) = \tau(g.c) = \tau(d)$ as claimed.
    
    On the other hand, if the $s$-colors of $c$ and $e$ lie in different $F_s^+$-orbits, then again by the definition of $\tau$, the chambers $\tau(c)$ and $\tau(e)$ are $s$-adjacent, and the $s$-color of $\tau(c)$ is $\lambda'_s(\tau(c)) = [\lambda_s(c)]$. Similarly, the chambers $\tau(g.c)$ and $\tau(g.e)$ are $s$-adjacent, and the $s$-color of $\tau(g.c)$ is $\lambda'_s(\tau(g.c)) = [\lambda_s(g.c)]$.
    Since the local action of $g$ on the $s$-panel through $c$ is contained in $F_s^+$, we have $[\lambda_s(g.c)] = [\lambda_s(c)]$.
    Combined with $\tau(e) = \tau(g.e)$, we conclude again that $\tau(c) = \tau(g.c) = \tau(d)$ as claimed.

    \medskip
        
    Conversely, assume that $\tau(c) = \tau(d)$ and consider a minimal gallery $\gamma$ from $c$ to~$d$,
    \[ c = c_0 \sim \dots \sim c_n = d, \]
    of type $s_1 s_2 \dotsm s_n$.
    Again, we proceed by induction on $n = \dist(c,d)$; there is nothing to prove if $n=0$. Suppose now that $n > 0$.
    
    Assume first that for some adjacent chambers $c_{i-1}$ and $c_i$ in this gallery, we have $\tau(c_{i-1}) = \tau(c_i)$.
    Then by definition of $\tau$, the chambers $c_{i-1}$ and $c_i$ are harmonious (w.r.t.\@~$\mathbf F^+$), i.e., their $s_i$-colors lie in a common $F_{s_i}^+$-orbit.
    We can hence find elements $g_1,\dots,g_\ell \in G$, each fixing some chamber in the same $s_i$-panel as $c_{i-1}$ and $c_i$, such that $g = g_1 \dotsm g_\ell$ maps $c_{i-1}$ to~$c_i$; in particular, $g \in G^+$.
    Now notice that we can form a shorter gallery
    \[ g.c = g.c_0 \sim \dots \sim g.c_{i-1} = c_i \sim \dots \sim c_n = d \]
    from $g.c$ to $d$ of length $n-1$, with $\tau(d) = \tau(c) = \tau(g.c)$ by the first part of the proof.
    By the induction hypothesis, $d$ and $g.c$ are in the same $G^+$-orbit, and since $g \in G^+$, we conclude that $c$ and $d$ are in the same $G^+$-orbit as well.
    
    We may hence assume that the gallery $\tau(\gamma)$ is non-stammering.
    In particular, each of the $s_i$ lies in $S'$, so that $c$ and $d$ lie in a common $S'$-residue of $\mathcal B$. Since the gallery $\gamma$ is minimal and since residues are convex, we infer that  $s_1 s_2 \dotsm s_n$ is a reduced word w.r.t.~$\Gamma'$, and hence $\tau(\gamma)$ is a minimal gallery from $\tau(c)$ to $\tau(d)$ (see \cite[Lemma~5.16]{AB08}).
    However, $\tau(\gamma)$ is a loop since $\tau(c) = \tau(d)$, so we get a contradiction.
\end{proof}

The following elementary observation will be used twice in the proof of \cref{thm:U(F/F+)}.
\begin{lem}\label{lem:H/H+}
    Let $H \leq \Sym(\Omega)$ be a permutation group and let $H^+$ be the normal subgroup generated by all point stabilizers.
    Consider the induced action of $H$ on the set $\Omega^0$ consisting of all $H^+$-orbits.
    If $g \in H$ fixes an element of $\Omega^0$, then $g \in H^+$. In particular, $H/H^+$ acts freely on $\Omega^0$.
\end{lem}
\begin{proof}
    Assume that $g \in H$ fixes an element of $\Omega^0$, i.e., stabilizes some $H^+$-orbit $H^+.x$ for some $x \in \Omega$.
    Then $x$ and $g.x$ lie in the same $H^+$-orbit, so $g.x = h.x$ for some $h \in H^+$.
Then $h^{-1}g$  fixes $x$, so $h^{-1}g \in H^+$ by the definition of $H^+$, and hence $g \in H^+$.
\end{proof}

We are now ready to state and prove the main result of this section. For a closely related result in the context of groups acting on trees, we refer to \cite[Theorem~4.12]{RST}. 

\begin{thm}\label{thm:U(F/F+)}
    The map $\psi \colon \mathcal U(\mathbf F) \to \mathcal U(\mathbf F/\mathbf F^+)$ given by $\psi(g).\tau(c) = \tau(g.c)$ for all $g \in \mathcal U(\mathbf F)$ and all $c \in \mathcal B$ is a group epimorphism, with kernel $\mathcal U(\mathbf F)^+$.
    In particular, $\mathcal U(\mathbf F)/\mathcal U(\mathbf F)^+ \cong \mathcal U(\mathbf F/\mathbf F^+)$.
\end{thm}
\begin{proof}
    As before, let $G = \mathcal U(\mathbf F)$ and $G^+ = \mathcal U(\mathbf F)^+ \unlhd G$.
    First observe that by Proposition~\ref{prop:G+-orbits}, $\psi$ is well-defined, i.e., if $\tau(c) = \tau(d)$ and $g \in G$, then $\tau(g.c) = \tau(g.d)$.
    Moreover, $\psi(g)$ preserves adjacency in~$\mathcal B^0$: if $c' \sim_s d'$ (for some $s \in S'$), then we can find representatives $c$ and $d$ in $\mathcal B$ with $\tau(c) = c'$ and $\tau(d) = d'$ such that $c \sim_s d$; then $g.c \sim_s g.d$, but certainly $\tau(g.c) \neq \tau(g.d)$ so we must have $\tau(g.c) \sim_s \tau(g.d)$.
    It is then obvious that $\psi$ is a group homomorphism from $G$ to $\Aut(\mathcal B^0)$.
    
    Next, we claim that $\psi(g) \in \mathcal U(\mathbf F/\mathbf F^+)$. So let $\mathcal P'$ be an arbitrary $s$-panel in $\mathcal B^0$ (for $s \in S')$. Then the local action of $\psi(g)$ on $\mathcal P'$ is the permutation of $Y_s$ induced by $\psi(g)_{|\mathcal P'} \colon \mathcal P' \to \psi(g).\mathcal P'$.
    On the other hand, we can lift $\mathcal P'$ to an $s$-panel $\mathcal P$ in $\mathcal B$ and by assumption, the local action $\sigma$ of $g$ on $\mathcal P$ is contained in $F_s$. In particular, $\sigma$ permutes the $F_s^+$-orbits, and the induced action on these orbits is precisely the image of $\sigma$ in $F_s/F_s^+$. By definition of~$\psi$, this is also the local action of $\psi(g)$ on~$\mathcal P'$, proving our claim.
    
    We now show that $\psi$ is surjective.
    First, we claim that the image $\psi(\mathcal U(\mathbf F))$ and $\mathcal U(\mathbf F/\mathbf F^+)$ have identical orbits on $\mathcal B^0$. Indeed, let $c' = \tau(c)$ and $d' = \tau(d)$ be two chambers in the same orbit of $\mathcal U(\mathbf F/\mathbf F^+)$. Then by \cite[Proposition 3.1]{BDM21},  their colors $\lambda'_s(c')$ and $\lambda'_s(d')$ lie in the same orbit of $F_s/F_s^+$, for all $s \in S'$.
    By definition of $\tau$, however, we have $\lambda'_s(c') = [\lambda_s(c)]$ and $\lambda'_s(d') = [\lambda_s(d)]$, hence the colors $\lambda_s(c)$ and $\lambda_s(d)$ lie in the same orbit of $F_s$, for all $s \in S'$. This obviously also holds for $s \in S \setminus S'$.
    By \cite[Proposition 3.1]{BDM21} again, $c$ and $d$ lie in the same orbit of $\mathcal U(\mathbf F)$, and we conclude that $c' = \tau(c)$ and $d' = \tau(d)$ lie in the same orbit of $\psi(\mathcal U(\mathbf F))$. This proves our claim.
    
    Next, we claim that the for any chamber $c'$ in $\mathcal B^0$, the chamber stabilizer $\mathcal U(\mathbf F/\mathbf F^+)_{c'}$ is trivial, so in particular, $\psi(\mathcal U(\mathbf F))_{c'}$ and $\mathcal U(\mathbf F/\mathbf F^+)_{c'}$ coincide.
    Indeed, by \cite[Lemma 3.16]{BDM21}, the local actions of $\mathcal U(\mathbf F/\mathbf F^+)_{c'}$ are contained in $(F_s/F_s^+)^+$, but by \cref{lem:H/H+}, this group is trivial. An element fixing a chamber and having trivial local actions must be trivial, proving our claim.
    Combined with the previous paragraph, this proves that $\psi$ is indeed surjective.
    
    Finally, the kernel of $\psi$ consists precisely of those $g \in G$ such that $\tau(g.c) = \tau(c)$ for all $c \in \mathcal B$.
    By \cref{prop:G+-orbits}, this holds if and only if $c$ and $g.c$ lie in the same $G^+$-orbit, for all $c \in \mathcal B$.
    Of course, if $g \in G^+$, then this holds; hence $G^+ \leq \ker \psi$.
    Conversely, assume that $g \in G$ stabilizes the $G^+$-orbit of $c$; then by \cref{lem:H/H+} again, $g \in G^+$. This shows that $\ker \psi = G^+$, finishing the proof.
\end{proof}

\section{Locally compact groups quasi-isometric to RAAGs}

\subsection{Blown-up buildings}\label{sec:blow-up}

We recall that two residues  $\mathcal R$ and $\mathcal R'$ in a building are called \textbf{parallel} if the projection maps $\proj_{\mathcal R}$ and $\proj_{\mathcal R'}$ define reciprocal bijections between the respective chamber sets of  $\mathcal R$ and $\mathcal R'$. In a right-angled building, parallel residues have the same type (see \cite[Proposition~5.37]{AB08}). 

Let $\mathcal B$ be a right-angled building of type $\Gamma$. For each $s \in S$, let $\Omega_s$ be a set. 
Following Huang--Kleiner \cite[Definition~5.6]{HK18}, we define a \textbf{blow-up data}  on $\mathcal B$ as a collection  of maps   $\mathscr H = \{h_\mathcal R \colon \mathcal R \to \Omega_s\}$,  for each $s \in S$ and for each $s$-panel $\mathcal R$, satisfying the following compatibility condition: if two panels $\mathcal R$ and $\mathcal R'$ are parallel, then $h_{\mathcal R'} = h_{\mathcal R} \circ \proj_{\mathcal R}$.  
In case $\Omega_s = \mathbf Z$ for all $s$, we say that $\mathscr H$ is a  \textbf{$\mathbf Z$-blow-up data}. 


\begin{remark}
Although  \cite[Definition~5.6]{HK18} focuses on $\mathbf Z$-blow-up data, the possibility to replace $\mathbf Z$ by more general sets  is  pointed out in \cite[Remark~5.12]{HK18}. We remark that legal colourings 
of a regular right-angled building may be viewed as  collections of maps defined on each panel, and satisfying the compatibility condition stated above. From this point of view,  
legal colourings are blow-up data where each blow-up map $h_{\mathcal R}$ is a bijection.
%
\end{remark}

Given a $\mathbf Z$-blow-up data $\mathscr H = \{h_\mathcal R \colon \mathcal R \to \mathbf Z\}$ on $\mathcal B$, we define a graph structure on $\mathcal B$, denoted by $|\mathcal B|_{\mathscr H}$, by declaring that two chambers $c, d \in \mathcal B$ are adjacent in  $|\mathcal B|_{\mathscr H}$ if $c, d$ belong to a common panel $\mathcal R$, and if $|h_{\mathcal R}(c) - h_{\mathcal R}(d)| \leq 1$. The graph $|\mathcal B|_{\mathscr H}$ is called the \textbf{blown-up chamber graph} associated with $\mathscr H$. 

Observe that if each map in the blow-up data is constant, then $|\mathcal B|_{\mathscr H}$ is nothing but the chamber graph $\mathscr C(\mathcal B)$. The identity on $\mathcal B$ defines a natural map $|\mathcal B|_{\mathscr H} \to \mathscr C(\mathcal B)$ mapping edges to edges.  

Another example is obtained by  viewing the legal colouring of $\mathcal B$ from Lemma~\ref{lem:legal-colouring} as a blow-up data. The  associated  blown-up chamber graph is nothing but the Cayley graph of   $G(\Gamma)$ with respect to the standard generating set, which we denote by $X(\Gamma)^{(1)}$ since it coincides with the $1$-skeleton of the CAT(0) cube complex $X(\Gamma)$ defined as the universal cover of the Salvetti complex (see e.g. \cite{HK18}). 

Let   $\mathscr H = \{h_\mathcal R \colon \mathcal R \to \mathbf Z\}$ be a  $\mathbf Z$-blow-up data on $\mathcal B$ and let  $H$ be a group acting on $\mathscr C(\mathcal B)$ by automorphisms. Then $H$ maps panels to panels, see \S\ref{sec:chamber-graph}. Following \cite[\S5.6]{HK18}, we say that $\mathscr H$ is \textbf{$H$-equivariant} if for each panel $\mathcal R$ and each $g \in H$, there exists an isometry $\rho(g, \mathcal R) \in \Isom(\mathbf Z)$ such that
$$h_{g(\mathcal R)} \circ g|_\mathcal R = \rho(g, \mathcal R) \circ h_\mathcal R$$
and satisfying the equality 
$$\rho(g_1 g_2, \mathcal R) = \rho(g_1, g_2(\mathcal R)) \circ \rho(g_2, \mathcal R)$$
for all $g_1, g_2 \in H$ and all $\mathcal R$. 

Notice the similarity with the notion of the \emph{local action}: the isometry $\rho(g, \mathcal R)$ plays the role of the local action of $g$ at $\mathcal R$.  Notice in particular that  the map $g \mapsto \rho(g, \mathcal R)$ is a homomorphism of the stabilizer $H_\mathcal R$ to $\Isom(\mathbf Z)$.

This terminology allows us to obtain the following combinatorial formulation of  \cite[Theorems~6.2]{HK18}.

\begin{thm}[Huang--Kleiner]\label{thm:HK}
Let $\Gamma$ be a graph  satisfying (R1) and (R2), and $\mathcal B$ be the right-angled building of type $\Gamma$ associated with the RAAG $G(\Gamma)$.  Let also $H$ be a group quasi-isometric to $G(\Gamma)$.

Then $H$ has a cobounded action by automorphisms on the chamber-graph $\mathscr C(\mathcal B)$. 
Moreover, there is an $H$-equivariant $\mathbf Z$-blow-up data  $\mathscr H = \{h_\mathcal R \colon \mathcal R \to \mathbf Z\}$ on $\mathcal B$, consisting of surjective maps with finite fibers of uniformly bounded cardinality,  so that $H$ acts by automorphisms on the blown-up chamber graph $|\mathcal B|_{\mathscr H}$, which is connected and locally finite. Moreover, the $H$-action on $|\mathcal B|_{\mathscr H}$ is cobounded and metrically proper, the orbit maps are quasi-isometries, and the natural map $|\mathcal B|_{\mathscr H} \to \mathscr C(\mathcal B)$ is $H$-equivariant. 
\end{thm}
\begin{proof}
The $H$-action on $\mathscr C(\mathcal B)$ is afforded by \cite[Theorem~1.3 and Lemma~3.15]{HK18}. The construction of the $H$-equivariant blow-up data  $\mathscr H = \{h_\mathcal R \colon \mathcal R \to \mathbf Z\}$ with the required properties is achieved in the proof of \cite[Theorem~6.2]{HK18}.  See, in particular, Lemma~\cite[Lemma~5.14]{HK18}  for the construction of the isometries $\rho(g, \mathcal R)$.  It follows that $H$ acts by automorphisms on the graph $|\mathcal B|_{\mathscr H}$. Moreover, the orbit maps of $H$ to $|\mathcal B|_{\mathscr H}$ are quasi-isometries. Hence the action is metrically proper and cobounded. 
\end{proof}


\subsection{Implosions}

Let $\mathcal B$ be a right-angled building of type $\Gamma$ 
 and  $\mathscr H = \{h_\mathcal R \colon \mathcal R \to \Omega'_s\}$ be a blow-up data. Let also $q'_s $ be the cardinality of $\Omega'_s$ and assume that $q'_s \geq 2$ for all $s$. (This assumption is not essential but simplifies the setup, and it will always be satisfied in our applications.)
 We let  $\mathcal B'$ be a semiregular building of type $\Gamma$ and of thickess $\mathbf q' =(q'_s)_{s \in S}$, and $\lambda' \colon \mathcal B' \to \prod_{s \in S} \Omega'_s$ be a legal colouring.

The following may be viewed as a generalization of \cite[Proposition~5.1]{BDM21}. In some sense, it shows that a blow-up data can also be viewed as a ``blow-down data.''

\begin{prop}\label{prop:implosion}
Given chambers $c_0 \in \mathcal B$ and $c'_0 \in \mathcal B'$, there is a unique map $\tau \colon \mathcal B \to \mathcal B'$ with $\tau(c_0) = c'_0$ such that 
for all $s \in S$ and each $s$-panel $\mathcal R$ in $\mathcal B$,  the set $\tau(\mathcal R)$ is contained in an $s$-panel of $\mathcal B'$ and 
$$\lambda'_s \circ \tau|_{\mathcal R} = h_\mathcal R.$$
Moreover, if $h_\mathcal R$ is surjective for all $\mathcal R$, then  $\tau$ is surjective.
\end{prop}
\begin{proof}
The proof that $\tau$ exists and satisfies the required properties mimics \cite[Proposition~5.1]{BDM21}, as follows.
We set $\tau(c_0) = c'_0$ and proceed to define $\tau(c)$ by induction on the numerical distance from $c_0$ to $c$.
Assume that $\tau$ has been defined for all chambers at distance $\leq n$ from $c_0$ and let $c$ be a chamber at distance $n+1$.
Let $d$ be a chamber at distance $n$ from $c_0$ that is $s$-adjacent to $c$ for some $s \in S$, and let $\mathcal R$ be the $s$-panel containing $c$ and $d$.
If $h_{\mathcal R}(c) = h_{\mathcal R}(d)$, we declare $\tau(c) := \tau(d)$.
If not, then we declare $\tau(c)$ to be the unique chamber $c'$ in $\mathcal B'$ that is $s$-adjacent to $\tau(d)$ with $\lambda'_s(c') = h_{\mathcal R}(c)$.

We claim that $\tau$ is well-defined.
Indeed, assume that $d_1$ and $d_2$ are two different chambers at distance $n$ from $c_0$ adjacent to $c$.
Then $d_1$ is $s$-adjacent to $c$ and $d_2$ is $t$-adjacent to $c$ for some $s,t \in S$ with $s \neq t$.
By the ``closing squares lemma'' \cite[Lemma 2.7(i)]{BDM21}, we must have $st = ts$ and there is a chamber $e$ at distance $n-1$ from $c_0$ that is $t$-adjacent to $d_1$ and $s$-adjacent to $d_2$.
In particular, the $s$-panels $\mathcal R_1$ through $c$ and $d_1$ and $\mathcal R_2$ through $e$ and $d_2$ are parallel.
The compatibility condition for blow-up data now tells us that $h_{\mathcal R_1}(c) = h_{\mathcal R_2}(d_2)$ and $h_{\mathcal R_1}(d_1) = h_{\mathcal R_2}(e)$.
Similarly, we have parallel $t$-panels $\mathcal R_3$ through $c$ and $d_2$ and $\mathcal R_4$ through $e$ and $d_1$, and we get
$h_{\mathcal R_3}(c) = h_{\mathcal R_4}(d_1)$ and $h_{\mathcal R_3}(d_2) = h_{\mathcal R_4}(e)$.
This implies that both possible paths starting at $\tau(e)$ end up in the same chamber of~$\mathcal B'$, so $\tau(c)$ does not depend on this choice.

We next show that if $c$ and $d$ are $s$-adjacent chambers in $\mathcal B$, then $\tau(c)$ and $\tau(d)$ are also $s$-adjacent. In the case where $d(c_0, c) \neq d(c_0, d)$, this follows directly from the definition of $\tau$. Let us assume that $d(c_0, c) = d(c_0, d)$.  Let $\mathcal R$ be the $s$-panel shared by $c$ and $d$. The assumption ensures that $x =\proj_{\mathcal R}(c_0)$ is different from $c$ and $d$ (recall that $\proj_{\mathcal R}(c_0)$ is the \emph{unique} chamber of $\mathcal R$ at minimal distance from~$c_0$). By the definition of $\tau$, it then follows that $\tau(x)$ and $\tau(c)$ are $s$-adjacent, and that $\tau(x)$ and $\tau(d)$ are also $s$-adjacent. Hence $\tau(c)$ and $\tau(d)$ are $s$-adjacent. This confirms that $\tau$ preserves the relation of $s$-adjacency for all $s \in S$. 

Suppose finally that  $h_\mathcal R$ is surjective for all $\mathcal R$. The fact that  $\tau$ is surjective follows from the definition since the building $\mathcal B'$ is connected, and any gallery $\gamma'$ in $\mathcal B'$ starting at $c_0'$ can be lifted to a gallery $\gamma$ in $\mathcal B$ starting at $c_0$ with $\tau(\gamma) = \gamma'$.
\end{proof}

The map $\tau$ is called an \textbf{implosion} of $\mathcal B$ with respect to given blow-up data.
We illustrate this construction with the following result (see \S\ref{sec:chamber-graph} for the notation). 

\begin{prop}\label{prop:implosion-to-Coxeter}
Let  $H \leq \widetilde{\Aut(\mathcal B)}$  and let $\mathscr H = \{h_\mathcal R \colon \mathcal R \to \mathbf Z\}$ be an $H$\dash equivariant  $\mathbf Z$-blow-up data on $\mathcal B$. Let also $\mathscr H' = \{h'_\mathcal R \colon \mathcal R \to \mathbf Z/2\mathbf Z\}$ be the blow-up data defined by 
$$h'_\mathcal R(c) = h_\mathcal R(c) +2\mathbf Z.$$ 
Let $c_0 \in \mathcal B$ and $c'_0 \in \mathcal B'$ and let $\tau \colon \mathcal B \to \mathcal B'$ be the implosion map satisfying $\tau(c_0) = c'_0$ afforded by Proposition~\ref{prop:implosion}.

Then   $\mathcal B'$ is the thin building associated with  the Coxeter group $W(\Gamma)$. Moreover, there is a homomorphism $\alpha \colon H \to \widetilde{\Aut(\mathcal B')} \cong W(\Gamma) \rtimes \Aut(\Gamma)$ such that $\tau \circ g = \alpha(g) \circ \tau$ for all $g \in H$.
\end{prop}

A very special case to keep in mind is the group $H = G(\Gamma)$, i.e. the case where $H$ is the RAAG of type $\Gamma$. In that case, Proposition~\ref{prop:implosion-to-Coxeter} affords the standard quotient map of $G(\Gamma)$ to the Coxeter group $W(\Gamma)$. The point of the proposition is to construct a suitable generalization of this quotient map, under the sole hypothesis that $\mathcal B$ carries an $H$-equivariant $\mathbf Z$-blow-up data. 

\begin{proof}[Proof of Proposition~\ref{prop:implosion-to-Coxeter}]
By definition, the building $\mathcal B'$ is right-angled of type $\Gamma$ and of constant thickness~$2$. Hence it is the thin building of type $\Gamma$. This implies that $ \widetilde{\Aut(\mathcal B')} \cong W(\Gamma) \rtimes \Aut(\Gamma)$. 

Fix $g \in H$. Let $\pi \colon \widetilde{\Aut(\mathcal B)}  \to \Aut(\Gamma)$ be the canonical projection. We need to construct $\alpha(g) \in \widetilde{\Aut(\mathcal B')}$ such that 
$$\tau \circ g = \alpha(g) \circ \tau.$$
Set $\tau(g(c_0)) = c''_0$ and $v = \delta(c'_0, c''_0) \in W(\Gamma)$,  where $\delta \colon \mathcal B' \times \mathcal B' \to W(\Gamma)$ is the Weyl distance.  By viewing the thin building $\mathcal B'$ as the Cayley graph of $W(\Gamma)$ based at the chamber $c'_0$, we fix an isomorphism  $\widetilde{\Aut(\mathcal B')} \cong W(\Gamma) \rtimes \Aut(\Gamma)$. This allows us to define  $\alpha(g)$ by identifying $\alpha(g)$ as the element $(v, \pi(g)) \in W(\Gamma) \rtimes \Aut(\Gamma)$. Equivalently, this means that $\alpha(g)$ is  the unique element of $\widetilde{\Aut(\mathcal B')}$ mapping $c'_0$ to $c''_0$, and  such that for all $x' \in \mathcal B'$ and  $w' \in W(\Gamma)$, we have 
$$\delta(c'_0, x') = w' \quad \Longrightarrow \quad \delta(c''_0, \alpha(g).x') = \pi(g)(w').$$

We now show   that for each chamber $x \in \mathcal B$, we have $\tau \circ g(x) = \alpha(g) \circ \tau(x)$.

Fix a minimal gallery 
$$c_0 \sim_{s_1} c_1  \sim_{s_2} \cdots  \sim_{s_n} c_n = x,$$ 
where $s_i \in S$ for all $i$, and set $w =s_1\dotsm s_n$. By \cite[Lemma~5.16]{AB08}, we have $\delta(c_0, x) = w$. 
For each $i$, let $\mathcal R_i$ be the $s_i$-panel shared by $c_{i-1}$ and $c_i$. We also define $s'_i \in S$ by setting 
$$s'_i =\left\{
 \begin{array}{cl}
 s_i & \text{if  } h_{\mathcal R_i} (c_{i-1}) \not \equiv h_{\mathcal R_i} (c_{i}) \mod 2\\
 1 & \text{if  } h_{\mathcal R_i} (c_{i-1})  \equiv h_{\mathcal R_i} (c_{i}) \mod 2.
 \end{array}
 \right.$$
 Set $x' = \tau(x)$ and $w' = s'_1 \dotsm s'_n \in W(\Gamma)$.   Since  $\tau$ maps $t$-panels to $t$-panels for all $t \in S$ by  Proposition~\ref{prop:implosion}, we infer that 
$$\big(\tau(c_0) = c'_0, \tau(c_1) ,  \dots  , \tau(c_n) = \tau(x) = x'\big)$$ 
is a sequence of consecutively adjacent chambers in the thin building $\mathcal B'$ such that $\tau(c_{i-1})$ and $\tau(c_i)$ are distinct and $s'_i$-adjacent if  $s'_i \neq 1$, and $\tau(c_{i-1}) = \tau(c_i)$ if $s'_i = 1$. Since $\mathcal B'$ is a thin building of type $\Gamma$, it follows that $\delta(c'_0, x') =  w'$ (see \cite[\S 3.5]{AB08}). 
%
%

Transforming the above gallery from $c_0$ to $x$ by $g$, we obtain a minimal gallery
$$g(c_0) \sim_{t_1} g(c_1)  \sim_{t_2} \cdots  \sim_{t_n} g(c_n) = g(x),$$
where $t_i = \pi(g)(s_i)$ for all $i$. We also define $t'_i \in S$ in a similar way as $s'_i$, by setting $t'_i = t_i$ if $h_{g(\mathcal R_i)} \circ g(c_{i-1}) \not \equiv h_{g(\mathcal R_i)} \circ g(c_{i}) \mod 2 $, and by setting $t'_i = 1$  otherwise. Again, using that the implosion $\tau$ maps panels to panels of the same type, it follows that 
$$\big(\tau\circ g(c_0) = c''_0 ,   \tau \circ g (c_1) , \dots ,  \tau \circ g(c_n) = \tau \circ g (x)\big)$$ 
is a sequence of consecutively adjacent chambers  joining $c''_0$ to $\tau \circ g(x)$ in $\mathcal B'$,  such that $\tau \circ g (c_{i-1})$ and $\tau \circ g (c_i)$ are distinct and $t'_i$-adjacent if  $t'_i \neq 1$, and $\tau \circ g (c_{i-1}) = \tau \circ g(c_i)$  if $t'_i = 1$. As above, since the building $\mathcal B'$ is thin, we deduce that $\delta(c''_0, \tau\circ g(x)) = t'_1 \dotsm t'_n$. 

We now recall that $\mathscr H$ is $H$-equivariant by hypothesis. 
Since every isometry of $\mathbf Z$ either preserves the even numbers, or swaps the even and odd numbers, it follows  that for all $i \in \{1, \dots, n\}$, we have 
$$ h_{g(\mathcal R_i)} \circ g(c_{i-1}) \equiv h_{g(\mathcal R_i)} \circ g(c_{i}) \mod 2 \quad \Longleftrightarrow \quad h_{\mathcal R_i} (c_{i-1}) \equiv h_{ \mathcal R_i}(c_{i}) \mod 2.$$
Therefore, we have $t'_i = t_i$ if and only if $s'_i = s_i$. It follows that $\pi(g)(s'_i)=t'_i$ for all $i$, so that $\pi(g)(w') =\pi(g)(s'_1\dotsm s'_n)= t'_1 \dotsm t'_n$. In view of the definition of $\alpha(g)$, this implies that 
 $\alpha(g)$ maps $x' = \tau(x)$ to $\tau \circ g(x)$, thereby confirming that $\tau \circ g(x) = \alpha(g) \circ \tau(x)$.

Since the latter equality holds for all $x \in \mathcal B$, we infer that $\alpha \colon H \to \widetilde{\Aut(\mathcal B')}$ is a group homomorphism. The conclusion follows.
%
%
%
\end{proof}

\subsection{Locally compact groups quasi-isometric to RAAGs}

Our next goal is to specialize  Huang--Kleiner's Theorem~\ref{thm:HK} to possibly non-discrete locally compact groups. We view a compactly generated locally compact group as a metric space, by endowing it with the word metric associated with a compact generating set. 

The main point of the following is that the resulting action on the blown-up chamber graph is   automatically continuous. This will follow from a continuity criterion due to Y.~Cornulier \cite[Lemma~19.29]{Cor18}.

\begin{prop}\label{prop:continuity}
Let $\Gamma$ be a finite graph satisfying (R1) and (R2), and $H$ be a compactly generated   locally compact group that is quasi-isometric to $G(\Gamma)$. Then the  $H$-actions on $\mathscr C(\mathcal B)$ and on the blown-up chamber graph $ |\mathcal B|_\mathscr H$, afforded by Theorem~\ref{thm:HK},  are  continuous. In particular, the identity component $H^\circ$ is compact and acts trivially.
\end{prop}
\begin{proof}
Recall that the vertices of  $\mathscr C(\mathcal B)$ and of the blown-up chamber graph $Y = |\mathcal B|_\mathscr H$ are the chambers of $\mathcal B$. It suffices to show that the $H$-action on $Y$ is continuous. 

To check continuity, we invoke the  criterion  afforded  by \cite[Lemma~19.29]{Cor18}. Since the graph $Y$ is connected and locally finite, its full automorphism group $\Aut(Y)$ is a tdlc group; its action on $Y$ is continuous. We claim that the only element of $\Aut(Y)$ acting on $Y$ with bounded displacement is the identity. The hypotheses (R1) and (R2) ensure that the type of $\mathcal B$ has no irreducible component of spherical or affine type. Hence, by the Rank Rigidity theorem from \cite{CapSag} it follows that no non-trivial automorphism of the Davis realization $|\mathcal B|$ has bounded displacement. Since the natural map $Y \to \mathcal B$ is $1$-Lipschitz, an automorphism acting with bounded displacement on $Y$ also has bounded displacement on $\mathcal B$, and must therefore be trivial. (Alternatively, we may also use a similar argument on the geometric realization of $Y$ constructed in \cite{HK18}, which is also a CAT(0) cube complex, hence subject to \cite{CapSag}.) 
By  Theorem~\ref{thm:HK}, the $H$-action on $Y$ is a cobounded action by automorphisms. Since compact subsets of $H$ are bounded in the word metric of $H$, they have bounded displacement in $Y$; thus the $H$-action on $Y$ is locally bounded. 
Therefore the hypotheses of \cite[Lemma~19.29]{Cor18} and we infer that the $H$-action on $Y$ is continuous. The properness of the action on $Y$ follows from the metric properness asserted by Theorem~\ref{thm:HK}. 

Since $\Aut(Y)$ is totally disconnected, the identity component $H^\circ$ acts trivially on~$Y$. Hence it is compact, since the action is proper. 
\end{proof}

Combining the previous results, we obtain the following.

\begin{thm}\label{thm:Coxeter-action}
Let $\Gamma$ be a finite graph satisfying  (R1),  (R2), (R3), and let    $H$ be a compactly generated locally compact group quasi-isometric to $G(\Gamma)$.

Then there is a continuous homomorphism $H \to W(\Gamma) \rtimes \Aut(\Gamma)$ whose image is of finite index. In particular, $H$ has a discrete quotient commensurable with $W(\Gamma)$.
\end{thm}
\begin{proof}
The statement is obvious if $\Gamma$ has only one vertex: indeed in that case $W(\Gamma)$  is finite, and the trivial action of $H$ satisfies the conclusions. We assume henceforth that $\Gamma$ has more than one vertex, hence $G(\Gamma)$ is not cyclic. 

By Theorem~\ref{thm:HK}, the group $H$  has a cobounded action by automorphisms on the chamber graph $\mathscr C(\mathcal B)$, where $\mathcal B$ is the regular right-angled building of type $\Gamma$ and countably infinite thickness. Since $\Gamma$ satisfies (R3), that action takes values in $\Aut(\mathcal B) \rtimes \Aut(\Gamma)$ by Lemma~\ref{lem:autom-chamber-graph}. Moreover, the action is continuous by Proposition~\ref{prop:continuity}. Theorem~\ref{thm:HK} also yields  an $H$-equivariant $\mathbf Z$-blow-up data $\mathscr H$.

Let $\mathcal B'$ be the right-angled thin building of type $\Gamma$, so that $\mathscr C(\mathcal B')$ is the Cayley graph of $W(\Gamma)$ with respect to its standard Coxeter generating set.   Proposition~\ref{prop:implosion-to-Coxeter} yields a homomorphism $H \to \widetilde{\Aut(\mathcal B')}  \cong W(\Gamma) \rtimes \Aut(\Gamma)$ such that the implosion map $\tau \colon \mathcal B \to \mathcal B'$ is equivariant. Since $H$ acts with open stabilizers on $\mathcal B$, its action on $\mathcal B'$ also has open stabilizers, hence it is continuous. 

Since $\tau$ is surjective and since the $H$-action on $\mathcal B$ is cobounded, it follows that the $H$-action  $\mathcal B'$ is cocompact. Hence the image of $H$ in $W(\Gamma) \rtimes \Aut(\Gamma)$ has finite index.
\end{proof}

\section{Lattice envelopes}

\subsection{\texorpdfstring{Envelopes of $2$-ended groups}{Envelopes of 2-ended groups}}

Let $X$ be a finitely generated $2$-ended group and $\Omega$ be a set on which $X$ acts regularly. The goal of this section is to describe all the lattice envelopes $Y$ of $X$ such that $X \leq Y \leq \Sym(\Omega)$. This will be achieved in Lemma~\ref{lem:env-2-ended}, which describes locally finite connected graph structures on  $\Omega$ that are automatically preserved by $Y$. 

The graphs in question belong to the following family. For each integer $n \geq 1$, let $\mathcal G_n$ be the  Cayley graph of 
$$Z_n = \mathbf Z/n\mathbf Z \times \mathbf Z$$
with respect to the generating set 
$$(\mathbf Z/n\mathbf Z \times \{0\}) \cup (\mathbf Z/n\mathbf Z \times \{1\}).$$
It is easy to see that $\Aut(\mathcal G_n)$ is isomorphic to the unrestricted wreath product
$$\Aut(\mathcal G_n) \cong \big( \prod_{\mathbf Z} \Sym(n) \big) \rtimes D_\infty,$$
where the action of $D_\infty$ on $\mathbf Z$ is the natural action occurring by viewing $D_\infty$ as the automorphism group of the simplicial line, whose vertex set is identified with~$\mathbf Z$. 
Clearly we have $\Aut(\mathcal G_1)  \cong D_\infty$.

Before stating the lemma, we recall that it follows from Stallings' theorem that every finitely generated $2$-ended group $X$ is virtually cyclic. Hence it has a largest finite normal subgroup $N$ such that $X/N$ is infinite cyclic or infinite dihedral. A similar result holds for in the non-discrete case: every compactly generated $2$-ended locally compact group $L$ has a largest compact normal subgroup $W$ such that $L/W$ is infinite cyclic or infinite dihedral, see  \cite[Proposition~5.6]{CCMT}.  If $L$ is a lattice envelope of $X$, the image of $X$ in $L/W$
 is a finite index subgroup. A key point for us is that, under the hypotheses of the next result, that index must be $1$ or $2$.

\begin{lem}\label{lem:env-2-ended}
Let $X$ be a finitely generated $2$-ended group and $L$ be a locally compact group containing $X$ as a cocompact lattice. Let also $\Omega$ be a discrete set and $\alpha \colon L \to \Sym(\Omega)$ be a continuous homomorphism such that the  restriction $\alpha|_X$ yields a regular action of $X$ on $\Omega$, and let $n$ be the order of the largest finite normal subgroup $N$ of  $X$. 

Then $L$ has a largest compact normal subgroup $W$, and exactly one of the following two cases occurs. 
\begin{enumerate}
\item All $W$-orbits on $\Omega$  have size $n$, and there is a bijection $\sigma \colon \Omega \to V\mathcal G_n$ such that $\sigma \alpha(L) \sigma^{-1} \leq \Aut(\mathcal G_n)$.

\item All $W$-orbits have size $2n$, the quotient $X/N$ is dihedral, and  there is a bijection $\sigma \colon \Omega \to V\mathcal G_{2n}$ such that $\sigma \alpha(L) \sigma^{-1} \leq \Aut(\mathcal G_{2n})$.
\end{enumerate} 
\end{lem}
\begin{proof}
Since the $X$-action on $\Omega$ is regular, we have a strict factorization $L = X L_\omega$, where $L_\omega$ is the stabilizer of a point $\omega \in \Omega$. By hypothesis $X$ is a cocompact lattice in $L$, hence $L_\omega$ is compact. In particular, $L$ acts properly on $\Omega$, and the kernel $K$ of the action is compact. Upon replacing $L$ by $L/K$ we shall assume henceforth that $X \leq L \leq  \Sym(\Omega)$.

Since $X$ is $2$-ended and $L$ is quasi-isometric to $X$, it follows that  $L$ is $2$-ended as well. Therefore it has a largest compact normal subgroup $W$ such that $L/W$ is infinite cyclic or infinite dihedral (see \cite[Proposition~5.6]{CCMT}). In particular,  the image of $X$ in $L/W$ is also  infinite cyclic or infinite dihedral, hence without any non-trivial finite normal subgroup. It follows that the   largest finite normal subgroup $N$ of $X$ is contained in $W$. Therefore we have $X \cap W  = N$.  

 Since $X/N $ is isomorphic to $\mathbf Z$  or $D_\infty$, there is an action $\rho$ of $X$ by automorphisms on the simplicial line $\mathbf Z$ such that $N$ acts trivially and $X/N$ acts regularly on the vertices. Indeed, we  may simply view the simplicial line $\mathbf Z$ as the Cayley graph of $X/N$ with respect to a standard generating set (a single generator if $X/N$ is cyclic, and a pair of involutions if $X/N$ is dihedral). By indexing the $N$-orbits on $\Omega$ accordingly, we obtain an $n$-to-$1$, $\rho$-equivariant, surjective map $f \colon \Omega \to \mathbf Z$. 


Since $N \leq W$ and since $W$ is normal in $L$, it follows that the $W$-orbits form an invariant partition of $\Omega$ that is coarser than the partition into $N$-orbits. Since $W$ is compact, its orbits are finite. Hence the image of the $W$-orbits under the map $f \colon \Omega \to \mathbf Z$  form a partition of the simplicial line into finite subsets, that is $X/N$-invariant.

Assume first that this is the trivial partition into singletons. In that case,  the  $N$-orbits and the $W$-orbits on $\Omega$ coincide. Therefore $L/W$ acts by permutations on the set $f(\Omega)$. Moreover, the restriction of that action to $XW/W \cong X/N$ is transitive.  Recall that $L/W$ is infinite cyclic or infinite dihedral. All transitive actions of that group on an infinite set have stabilizers of order~$1$ or $2$. 
It follows that the  group $XW/W \cong X/N$ has index $1$ or $2$ in  $L/W$.  In either case, it follows that the $L/W$-action on the set $f(\Omega) = \mathbf Z$  preserves the simplicial line structure on $\mathbf Z$, since the full automorphism group of $X/N$ acts on each of its Cayley graphs with respect to its standard generating sets. Thus the $X$-action on the simplicial line extends to an $L$-action, and the  $X$-equivariant map $f\colon \Omega \to \mathbf Z$ is also $L$-equivariant. 
We now define a graph structure on $\Omega$, denoted by $\mathcal G$,  by declaring that two points $v, w \in \Omega$ are adjacent if $|f(v) - f(w)| \leq 1$. 
It follows that  the graph $\mathcal G$  on $\Omega$ is preserved by $L$. Moreover $\mathcal G$ is isomorphic to $\mathcal G_n$. This finishes the proof in this case. 

Assume now that  $X/N$ is infinite cyclic. Then  the only partition of the simplicial line  into finite subsets that is $X/N$-invariant  is the trivial partition into singletons. Thus we are reduced to the previous case. 

We assume henceforth that $X/N$ is dihedral. 
Since the non-trivial finite subgroups of $X/N$ have order~$2$, the only invariant partitions for the $X$-action on the simplicial line $f(\Omega) = \mathbf Z$ into finite subsets are the trivial partition into singletons, and the partition into cosets of a fixed subgroup of order~$2$. If the partition of $f(\Omega)$ into $W$-orbits is trivial, we are once again reduced to the first case above. It remains to consider the case where the partition consists of subsets of size~$2$, which correspond to cosets of a subgroup of order $2$ of $X/N$, say $\langle \tau\rangle$. Let $\tau' $ be another involution of $X/N$, such that $\{\tau, \tau'\}$ forms a generating pair of the dihedral group $X/N$. Upon redefining the equivariant map $f \colon \Omega \to \mathbf Z$ if necessary by viewing the line  as a Cayley graph of $X/N$ with respect to the generating pair $\{\tau, \tau'\}$, we may assume that $f$ maps the partition into $W$-orbits  to the partition   $\mathbf Z = \bigcup_k \{2k, 2k+1\}$. 

Now the $L$-action on the $W$-orbits factors through $L/W$, which contains the dihedral group $XW/W \cong X/N$. In particular $L/W$ cannot be cyclic, hence it is dihedral. The $L/W$-action on the orbit space $W\backslash \Omega$ is thus a proper action of the  infinite dihedral group, whose restriction to $X$ is transitive. Since the stabilizer of a fixed $W$-orbit in $X$ has order~$2$ in the case at hand, we infer that the projection  of $X$ to the quotient $L/W$ is surjective. Hence $L = XW$. We then define a graph structure $\mathcal G'$ on $\Omega$, defined by declaring $v, w \in \Omega$ to be adjacent if there is $k \in \mathbf Z$ such that $\{ f(v), f(w) \} \subset [2k, 2k+3]$. Observe that $X$ and $W$ both preserve the graph $\mathcal G'$, hence so does the group $L$. Moreover $\mathcal G'$ is isomorphic to $\mathcal G_{2n}$. This concludes the proof.\end{proof}
	
The locally compact groups 
$$Y = \sigma^{-1} \Aut(\mathcal G_n) \sigma \leq \Sym(\Omega) \qquad \text{and} \qquad Y'= \sigma^{-1} \Aut(\mathcal G_{2n}) \sigma \leq \Sym(\Omega)$$ 
afforded by Lemma~\ref{lem:env-2-ended} are lattice envelopes of $X$ contained in $\Sym(\Omega)$. We call them the \textbf{canonical envelopes} of $X$. The lemma shows that indeed, those are natural envelopes of $X$ contained in $\Sym(\Omega)$. 

Remark that we could construct larger envelopes by replacing the interval $[2k, 2k+3]$ by a larger interval at the end of the proof of the lemma. Our choice was  made to obtain the simplest possible graph structure on $\Omega$, enjoying  the property that the quotient graph on the orbit space $W\backslash \Omega$ is the simplicial line.

The following observation provides a dual point of view: we fix the graph $\mathcal G_m$ and characterize the subgroups of its automorphism group acting regularly on its vertex set. 

\begin{lem}\label{lem:lattices-in-Aut(Gm)}
Let $m >0$ and let $X \leq \Aut(\mathcal G_m)$ be a group acting regularly on the vertex set of $\mathcal G_m$. Then $X$ is a finitely generated $2$-ended group. Denoting by $N$ its largest finite normal subgroup , then either  $m=|N|$, or  $X/N$ is dihedral and $m = 2|N|$.
\end{lem}
\begin{proof}
Let $\Omega$ be the vertex set of $\mathcal G_m$. The hypotheses imply that $X$ acts geometrically on the $2$-ended graph $\mathcal G_m$.  Hence $X$ is finitely generated and $2$-ended, and $X$ is a cocompact   lattice in   $L = \Aut(\mathcal G_m)$. 
By the definition of $L$, its largest compact normal subgroup $W$   acts on $\Omega$  with orbits of size~$m$. Hence the conclusion follows from  Lemma~\ref{lem:env-2-ended}.
\end{proof}
	
\subsection{Reducing to  totally disconnected cocompact envelopes}

The following result follows from the work of Bader--Furman--Sauer on lattice envelopes.

\begin{prop}\label{prop:envelope->coco:general}
Let $\Lambda$ be a countable group and $H$ be a locally compact group containing $\Lambda$ as a lattice. Suppose that $\Lambda$ satisfies the following conditions:
\begin{enumerate}
\item $\Lambda$ is acylindrically hyperbolic. 
\item $\Lambda$ is linear, or $\Lambda$ is finitely generated and residually finite. 
\item $\Lambda$ is virtually torsion-free. 
\item $\Lambda$ is not relatively hyperbolic with respect to virtually nilpotent subgroups. 
\end{enumerate}
Then the identity component $H^\circ$ is compact,  and the lattice $\Lambda$ is cocompact in $H$.
\end{prop}
\begin{proof}
We claim that $\Lambda$ satisfies the hypotheses of  \cite[Theorem~A]{BFS}, which are denoted  (Irr), (CAF), (NbC) and (BT) in loc.~cit.
Since $\Lambda$ is acylindrically hyperbolic, the properties (Irr) and  (CAF) hold by \cite[Theorem~1.4(3) and (4$'$)]{BFS}. Since $\Lambda$ is virtually torsion-free, property (BT) holds as well. Moreover, the hypothesis that $\Lambda$ is  linear implies that property (NbC) holds by \cite[Theorem~1.4(2$'$)]{BFS}, whereas if $\Lambda$ is finitely generated and residually finite, then the same conclusion holds by \cite[Theorem~1.4(3$'$)]{BFS} since (CAF) holds by the above. 

Thus the trichotomy from \cite[Theorem~A]{BFS} applies. Suppose, by contradiction, that the identity component $H^\circ$ is not compact. Then either $\Lambda$ is   virtually isomorphic to an irreducible lattice in a connected semisimple Lie group, or to an irreducible $S$-arithmetic lattice in a product of simple Lie and algebraic groups. Since $\Lambda$ is acylindrically hyperbolic, it is SQ-universal, while irreducible lattices in products of semisimple Lie and algebraic groups of  rank~$\geq 2$ satisfy Margulis' Normal Subgroup Theorem. The only remaining possibility is that $\Lambda$ is virtually isomorphic to a lattice in a connected simple Lie group of rank~$1$. In view of  \cite[Theorem~5.1]{Fa98}, we know that (non-)uniform lattices in simple Lie group of rank~$1$ are (relatively) hyperbolic (with respect to the so-called cusp groups, which are virtually nilpotent). This is not the case of $\Lambda$ by hypothesis. This shows  that $H^\circ$ is compact.

The fact that $\Lambda$ is cocompact follows from \cite[Theorem~A]{BFS}, since $\Lambda$ satisfies (BT).
\end{proof}

The following consequence now follows from known properties of RAAGs. For a related statement  under different assumptions on $\Gamma$, we refer to \cite[Theorem~10.1]{HH23}. 

\begin{prop}\label{prop:envelope->coco:RAAG}
Let $\Gamma$ be a finite connected graph with more than one vertex. Assume that $\Gamma$  is irreducible. Let also $\Lambda \leq G(\Gamma)$ be a finite index subgroup, and let $H$ be a locally compact group containing $\Lambda$ as a lattice. Then the identity component $H^\circ$ is compact,   the lattice $\Lambda$ is cocompact in $H$, and $\Lambda$ embeds as a cocompact lattice in the totally disconnected group   $H/H^\circ$.
\end{prop}
\begin{proof}
It suffices to show that $\Lambda$ satisfies the hypotheses of Proposition~\ref{prop:envelope->coco:general}. Since all those conditions are inherited by finite index subgroups, it suffices to show that the right-angled Artin group $G(\Gamma)$ satisfies them. 

By hypothesis $G(\Gamma)$ is irreducible but not virtually cyclic. It follows 
 from the Rank Rigidity Theorem for CAT(0) cube complexes that $G(\Gamma)$ is acylindrically hyperbolic, see \cite{CapSag} and \cite[\S8(d)]{Osin}. Moveover RAAGs are linear (this well-known fact follows for example from \cite{DavJan} together with the fact that Coxeter groups are linear). Since $G(\Gamma)$ act freely on a CAT($0$) space $X(\Gamma)$, it is torsion-free. Since the graph $\Gamma$ is connected with more than one vertex, it follows from  \cite[Proposition~1.3]{BDM09} that $G(\Gamma)$ is not relatively hyperbolic.
 
Hence $H^\circ$ is compact and $\Lambda$ is cocompact. Since $\Lambda$ is torsion-free, the intersection $H^\circ \cap \Lambda$ is trivial, so that  the canonical projection maps $\Lambda$ injectively onto a lattice in  $H/H^\circ$.
\end{proof}

\begin{remark}\label{rem:envelope-product-groups}
It turns out that Proposition~\ref{prop:envelope->coco:RAAG} remains true for finitely generated groups $\Lambda$  of the form $\Lambda \cong \Lambda_1 \times \dots \times \Lambda_k$, where each direct factor $\Lambda_i$ satisfies the conditions (i), (ii), (iii), (iv) from Proposition~\ref{prop:envelope->coco:general}. This can be established by adapting the proof of \cite[Theorem~A]{BFS}. In particular Proposition~\ref{prop:envelope->coco:RAAG} remains true if $\Gamma$ is a join of  irreducible  graphs, each of which has more than one vertex. If the graph $\Gamma$ satisfies moreover (R1) and (R2), then the conclusions of Proposition~\ref{prop:envelope->coco:RAAG}  also follow from   \cite[Theorem~10.1]{HH23} even if $\Gamma$ is not irreducible. This is sufficient for our purposes; therefore, we do not include the details of proof of the generalization of Proposition~\ref{prop:envelope->coco:RAAG}  for product groups.
\end{remark}

\section{\texorpdfstring{Graph products of $2$-ended groups and their envelopes}{Graph products of 2-ended groups and their envelopes}}

\subsection{\texorpdfstring{Graph products of $2$-ended groups are commensurable with RAAGs}{Graph products of 2-ended groups  are commensurable with RAAGs}}

For each $s \in S$, let $X_s $ be a $2$-ended group. Let $\mathbf X = (X_s)_{s\in S}$ and  
$$G_\Gamma(\mathbf X) = \bigast_{s \in S} X_s \big\slash \langle \! \langle [X_s, X_t]  : s, t \in S \text{ with } \{s, t\} \in E(\Gamma)\rangle \! \rangle$$ 
be the graph product of those groups $(X_s)_{s \in S}$ associated with the graph $\Gamma$. The following result is a special case of \cite[Theorem~1]{JS01}; we include a proof as it provides a simple illustration of the use of universal groups with prescribed local actions. 

\begin{prop}\label{prop:commensurable}
The group $G_\Gamma(\mathbf X) $ is abstractly commensurable with the right-angled Artin group $G(\Gamma)$. 
\end{prop}
\begin{proof}
For each $s$, let   $\Omega_s$ be a  countably infinite set. Let us fix a regular action of $X_s$ on the set $\Omega_s$ and use this action to view $X_s$ as a subgroup of $\Sym(\Omega_s)$. It then follows from Proposition~\ref{prod:graph-products-are-univ-groups} that the discrete group $G_\Gamma(\mathbf X)$ is isomorphic to the universal group $\mathcal U(\mathbf X)$.

Since $X_s$ is $2$-ended, it contains an infinite cyclic subgroup $T_s$ that centralizes the largest finite normal subgroup of $X_s$. The group $T_s$ has finite index in $X_s$; hence it acts freely with finitely many orbits on $ \Omega_s$. Let $n_s$ be the number of $T_s$-orbits. Then there exists a cyclic group $F_s \leq \Sym(\Omega_s)$ acting regularly on $\Omega_s$, and  containing $T_s$ as subgroup of index~$n_s$. Set $\mathbf T = (T_s)_{s \in S}$ and $\mathbf F = (F_s)_{s \in S}$. 

By \cref{prod:graph-products-are-univ-groups}, the universal group $\mathcal U(\mathbf F)$ is isomorphic to $G(\Gamma)$. By \cite[Lemma~3.3]{BDM21}, we have $\mathcal U(\mathbf T) \leq \mathcal U(\mathbf F) \cap \mathcal U(\mathbf X)$. Both groups $\mathcal U(\mathbf F)$ and $\mathcal U(\mathbf X)$ act regularly on the chambers of $\mathcal B$. Moreover, since $T_s$ has finitely many orbits on $\Omega_s$ for all $s$, it follows from \cite[Proposition~3.1]{BDM21} that $\mathcal U(\mathbf T)$ has finitely many orbits of chambers on the building $\mathcal B$. It follows that $\mathcal U(\mathbf T)$ has finite index in $\mathcal U(\mathbf F)$ and in $\mathcal U(\mathbf X)$. This confirms that the latter two groups are commensurable.  
\end{proof}

Observe that if the $2$-ended group $X_s$ is infinite dihedral for each $s$, then $G_\Gamma(\mathbf X)$ is a right-angled Coxeter group. In particular, Proposition~\ref{prop:commensurable} directly implies the following well-known result of Davis--Januszkiewicz:

\begin{cor}[\cite{DavJan}]\label{cor:DavJan}
Every right-angled Artin group is commensurable with a right-angled Coxeter group. 
\end{cor}

We now construct a larger envelope for  $G_\Gamma(\mathbf X) $ by passing from discrete $2$-ended  groups to their canonical envelopes  on the level of the local action. In view of Proposition~\ref{prod:graph-products-are-univ-groups}, we may  identify  $G_\Gamma(\mathbf X)$ with the universal group $\mathcal U^\lambda_{\mathcal B}(\mathbf F) \leq \Aut(\mathcal B)$, where $\lambda$ is the regular colouring from Lemma~\ref{lem:legal-colouring}.

\begin{prop}\label{prop:LattEnv}
For each $s \in S$, let  $X_s \leq \Sym(\Omega_s)$ be a $2$-ended group acting regularly on $\Omega_s$. Let also $Y_s, Y'_s \leq \Sym(\Omega_s)$ be its canonical envelopes, and let $L_s \in \{Y_s, Y'_s\}$ (in particular, we must have $L_s = Y_s$ if the quotient of $X_s$ by its largest finite normal subgroup is cyclic). Set $\mathbf X = (X_s)_{s \in S}$ and $\mathbf L = (L_s)_{s \in S}$. 

Then the group $\mathcal U(\mathbf L) = \mathcal U^\lambda_{\mathcal B}(\mathbf L)$  is a tdlc group containing   $G_\Gamma(\mathbf X)=  \mathcal U^\lambda_{\mathcal B}(\mathbf X)$ as a cocompact lattice. Moreover, $\mathcal U(\mathbf L)$ is non-discrete unless $\Gamma$ is the complete graph. 
\end{prop}

\begin{proof}
That $\mathcal U(\mathbf L) $ is a  tdlc group follows from \cite[Propositions~3.11]{BDM21}. 
This also ensures that the stabilizer of every chamber is a compact open subgroup of $\mathcal U(\mathbf F) $. 

Moreover, by definition the permutation group $L_s$ does not act freely on $\Omega_s$. Therefore, as soon as $\mathcal B$ is of non-spherical type (i.e., the graph $\Gamma$ is not the complete graph), it follows from \cite[Propositions~3.14]{BDM21}  that $\mathcal U(\mathbf L) $ is non-discrete.
(The hypothesis that $\Gamma$ is non-spherical is not explicitly stated in \cite[Propositions~3.14]{BDM21}, but it is indeed necessary. It is used in the proof of that proposition through \cite[Lemma~3.6]{BDM21}. More precisely, if $L_s$ does not act freely on $\Omega_s$ and some $t \in S$ does not commute with $s$, then $\mathcal U(\mathbf L)$ is non-discrete. If $\Gamma$ is non-spherical, then this holds for \emph{some} $s \in S$.)

Since $X_s \leq L_s$ for all $s$, we have $\mathcal U(\mathbf X) = \mathcal U^\lambda_{\mathcal B}(\mathbf X) \leq \mathcal U(\mathbf L) $ by \cite[Lemma~3.3]{BDM21}. Since $\mathcal U(\mathbf X)$ acts regularly on $\mathcal B$, we have a strict factorization $\mathcal U(\mathbf L)  = \mathcal U(\mathbf X)  K$, where $K$ is the stabilizer of a chamber in $\mathcal U(\mathbf L) $. Since $K$ is a compact open subgroup, it directly follows that $\mathcal U(\mathbf X)$ is a cocompact lattice in $\mathcal U(\mathbf L)$.
\end{proof}

In the special case where $X_s$ is infinite cyclic for each $s$, we have $G_\Gamma(\mathbf X)  = G(\Gamma)$ and $L_s$ is then the infinite dihedral group for each $s$, with its standard action on $\mathbf Z$ by automorphisms. In that case the group $\mathcal U(\mathbf  L)$ from Proposition~\ref{prop:LattEnv} is called the \textbf{universal group with infinite dihedral local action} associated with the right-angled building $\mathcal B$.  
%
%

\subsection{From lattice envelopes to universal groups}

The following straightforward observation will allow us to invoke Proposition~\ref{prop:envelope->coco:RAAG}.
\begin{lem}\label{lem:connected}
If $\Gamma$ satisfies (R1) and (R2), then $\Gamma$ is connected. 
\end{lem}
\begin{proof}
Let $S = S_1 \cup S_2$ be a partition of the vertex such that $S_1$ is a connected component of $\Gamma$.   If two vertices $s, t \in S_2$ are not adjacent, then the complement of $\{s\} \cup s^\perp$ has at least two connected components, one containing  $t$ and the other being $S_1$. Thus $S_2$ is a clique by (R1). Then (R2) implies that $S_2$ contains at most one element.  In that case, if $s \in S_2$, then $s^\perp$ is empty,  so (R2) fails unless $S_2$ is empty. This confirms that $S = S_1$. 
\end{proof}

The following shows that, under suitable assumptions on $\Gamma$, the lattice envelopes afforded by  Proposition~\ref{prop:LattEnv} are universal. 

\begin{thm}\label{thm:Main-latt-env}
Let $\Gamma$ be a finite graph,  satisfying  the conditions (R1), (R2) and (R3). Suppose that there is no vertex $s \in S$ such that $S \subset s \cup s^\perp$ (i.e. the complement of $\Gamma$ has no isolated node). 
Let $G_\Gamma(\mathbf X)$ be a graph product of discrete $2$-ended groups. Let $\mathcal B$ be the regular right-angled building of type $\Gamma$ and countably infinite thickness. For each $s$, let $\Omega_s$ be a set with a regular action of $X_s$ and $\lambda \colon \mathcal B \to \prod_{s \in S} \Omega_s$ be the legal colouring afforded by Lemma~\ref{lem:legal-colouring}.

For each  locally compact group $H$ containing a lattice isomorphic to $G_\Gamma(\mathbf X)$,   there is a continuous homomorphism 
$$\alpha \colon H \to \widetilde{\Aut(\mathcal B)} \cong \Aut(\mathcal B) \rtimes \Aut(\Gamma)$$ 
with compact kernel and closed image, 
whose restriction to $G_\Gamma(\mathbf X)$ is an isomorphism  $ G_\Gamma(\mathbf X) \to \mathcal U^\lambda_{\mathcal B}(\mathbf X) $, 
such that 
$$ U^\lambda_{\mathcal B}(\mathbf X)  \leq 
 \alpha(H)
\leq \mathcal U^{\lambda}_{\mathcal B}(\mathbf L) \rtimes \Aut(\Gamma),$$ where $\mathbf L = (L_s)_{s \in S}$  and for each $s$, the group $L_s \leq \Sym(\Omega_s)$ is a canonical envelope of the $2$-ended group $X_s \leq \Sym(\Omega_s)$ (see~\cref{lem:env-2-ended}).
\end{thm}
\begin{proof}
In view of Lemma~\ref{lem:connected}, we may invoke Proposition~\ref{prop:envelope->coco:RAAG} (see Remark~\ref{rem:envelope-product-groups} in case $\Gamma$ is not irreducible). This ensures that  $H^\circ$ is compact and that $G_\Gamma(\mathbf X)$ is cocompact in $H$. In particular, $H$ is quasi-isometric to $G_\Gamma(\mathbf X)$. 

Recall from Proposition~\ref{prod:graph-products-are-univ-groups} that $G_\Gamma(\mathbf X)$ is isomorphic to the universal group $\mathcal U^\lambda_{\mathcal B}(\mathbf X) $, where $\lambda$ is the legal colouring afforded by Lemma~\ref{lem:legal-colouring}. We identify these groups.  We also identify $G(\Gamma) $ and $ \mathcal U^\lambda_{\mathcal B}(\mathbf X') $, where 
 for each $s \in S$, the group $X'_s \leq \Sym(\Omega_s)$  is infinite cyclic acting regularly  on $\Omega_s$. 
By (the proof of) Proposition~\ref{prop:commensurable}, the groups $G_\Gamma(\mathbf X) = \mathcal U^\lambda_{\mathcal B}(\mathbf X) $ and $G(\Gamma)= \mathcal U^\lambda_{\mathcal B}(\mathbf X') $ have a common finite index subgroup $\mathcal U(\mathbf T)$. Both groups act regularly on $\mathcal B$. We fix a base chamber $c_0 \in \mathcal B$, and we let   $\phi \colon G_\Gamma(\mathbf X) \to G(\Gamma)$ be the unique map such that $g.c_0 = \phi(g).c_0$ for all $g \in  G_\Gamma(\mathbf X) $. Thus $\phi$ is a bijection that restricts to the identity  on $\mathcal U(\mathbf T)$. Moreover,  $\phi$ is left-$\mathcal U(\mathbf T)$-equivariant. It follows that  $\phi$ is a bijective quasi-isometry. 

In particular, $H$ is quasi-isometric to $G(\Gamma)$. 
We now  invoke Theorem~\ref{thm:HK} and Proposition~\ref{prop:continuity}, which afford a continuous homomorphism with compact kernel $\alpha \colon H \to  \Aut(\mathscr C(\mathcal B))$. 
Since $\Gamma$ satisfies (R3), we deduce from Proposition~\ref{prop:types} that $\alpha$ takes its values in $\widetilde{\Aut(\mathcal B)} \cong \Aut(\mathcal B) \rtimes \Aut(\Gamma)$. 

Since the quasi-isometry $\phi \colon G_\Gamma(\mathbf X) \to G(\Gamma)$ is bijective, it follows from the uniqueness part in \cite[Theorem~1.3]{HK18} that the restriction of $\alpha$ to $G_\Gamma(\mathbf X) =  \mathcal U^\lambda_{\mathcal B}(\mathbf X) $ coincides the regular action of that group on $\mathcal B$. In particular, the image $\alpha(H)$ contains $\mathcal U^\lambda_{\mathcal B}(\mathbf X) $, hence it is contained in a universal group $\mathcal U^\lambda_{\mathcal B}(\mathbf Y) $ by Proposition~\ref{prod:graph-products-are-univ-groups}(ii), where $Y_s \leq \Sym(\Omega_s)$ is the permutation group defined by the local action of the stabilizer of the $s$-panel $\mathcal P_s$ of $H$ containing a base chamber $c_0 \in \mathcal B$.  

Since the $H$-action on $\mathcal B$ is continuous, the stabilizer $\Stab_H(\mathcal P_s)$ is open in $H$. Therefore  $\Stab_{G_\Gamma(\mathbf X)}(\mathcal P_s)$ is a lattice in $\Stab_H(\mathcal P_s)$ (see \cite[\S2.C]{CapMon-KM}). Moreover, since the restriction of $\alpha$ to $ G_\Gamma(\mathbf X) = \mathcal U^\lambda_{\mathcal B}(\mathbf X) $ is the standard regular action of that group on $\mathcal B$, it follows that  $\Stab_{G_\Gamma(\mathbf X)}(\mathcal P_s)$ is a discrete $2$-ended group isomorphic to $X_s$, and acting  regularly on $\mathcal P_s$. In view of Lemma~\ref{lem:env-2-ended}, we deduce that there is a canonical envelope $L_s$ of $X_s$ in $\Sym(\Omega_s)$  
such that 
$$  \sigma(g, \mathcal P_s)   \in L_s$$
for all $g \in   \Stab_H(\mathcal P_s)$. 
By Proposition~\ref{prod:graph-products-are-univ-groups}(ii), we have
$$\alpha(H) \leq \mathcal U^{\lambda}_{\mathcal B}(\mathbf L) \rtimes \Aut(\Gamma).$$
%
%
It follows that 
$$ U^\lambda_{\mathcal B}(\mathbf X)  \leq   \alpha(H)   \leq \mathcal U^{\lambda}_{\mathcal B}(\mathbf L) \rtimes \Aut(\Gamma),$$
as required.
\end{proof}

%


\begin{cor}\label{cor:discrete-quot}
Every lattice envelope of $G_\Gamma(\mathbf X)$ has an infinite discrete quotient that is commensurable with $W(\Gamma)$. In the special case where $G_\Gamma(\mathbf X) = G(\Gamma)$, that quotient has a finite index subgroup isomorphic to $W(\Gamma)$. 
\end{cor}
\begin{proof}
Let $H$ be a lattice envelope of $G_\Gamma(\mathbf X)$. 
Theorem~\ref{thm:Main-latt-env} provides a continuous homomorphism
$$H \to  \mathcal U(\mathbf L) \rtimes \Aut(\Gamma),$$
whose image of the first map contains $G_\Gamma(\mathbf X)$. 

For each $s \in S$, the subgroup $L^+_s$ of $L_s$ generated by point stabilizers has two orbits on $\Omega_s$, and  the quotient $L_s/L^+_s$ is cyclic of order~$2$. It follows from Theorem~\ref{thm:U(F/F+)} that the quotient $ \mathcal U(\mathbf L)/ \mathcal U(\mathbf L)^+$ is isomorphic to $W(\Gamma)$. Therefore we obtain a homomorphism 
$$\mathcal U(\mathbf L) \rtimes \Aut(\Gamma) \to W(\Gamma)  \rtimes \Aut(\Gamma)$$
with open kernel. In particular, its   restriction   to  the lattice $G_\Gamma(\mathbf X)$ maps to a lattice, i.e., a finite index subgroup of  $W(\Gamma)$. Thus $H$ has a discrete quotient which is commensurable with   the infinite Coxeter group $W(\Gamma)$.

If $G_\Gamma(\mathbf X) = G(\Gamma)$, then the restriction of the homomorphism $H \to W(\Gamma)  \rtimes \Aut(\Gamma)$ to $G(\Gamma)$ is the standard quotient map $G(\Gamma) \to W(\Gamma)$. Thus the image of $H$ contains $W(\Gamma)$.
\end{proof}

As a consequence of 		
Theorem~\ref{thm:Main-latt-env}, we obtain the following generalization of  \cite[Theorem~1.9]{HK18}. 

\begin{cor}\label{cor:different-envelopes}
Let $\Gamma$ be a finite irreducible graph with more than one vertex satisfying (R1), (R2) and (R3). Then there exist an infinite family  of groups $(G_n)$, that are all commensurable with $G(\Gamma)$, such that for all $m \neq n$, the group $G_m$ and $G_n$ do not have a common lattice envelope. 
\end{cor}
\begin{proof}
For $n \geq 1$, we define $G_n$ as the graph product over $\Gamma$ of the $2$-ended group $Z_n = \mathbf Z/n\mathbf Z \times \mathbf Z$. 

Suppose that $G_m$ and $G_n$ have a common lattice envelope $H$. Let $\alpha \colon H \to \Aut(\mathcal B) \rtimes \Aut(\Gamma)$ be the  continuous proper homomorphism obtained from applying  Theorem~\ref{thm:Main-latt-env} to  the lattice embedding $G_m \leq H$. It then follows from \cite[Theorem~1.3]{HK18} that the restriction of $\alpha$ to $G_n$ is conjugate by an element of $\Aut(\mathscr C(\mathcal B))$ to the standard regular action of $G_n$ on $\mathcal B$ (as in section~\ref{ss:RABs}). Therefore, considering a panel $\mathcal P$, it follows that $\Stab_{G_n}(\mathcal P) \cong X_n$ and $\Stab_{G_m}(\mathcal P) \cong X_m$ both act regularly on $\mathcal P$. In view of Lemma~\ref{lem:lattices-in-Aut(Gm)} and Theorem~\ref{thm:Main-latt-env}, this implies that  $m = n$. 
\end{proof}
	
%

\subsection{The automorphism group of the Cayley graph of a RAAG}		

\begin{proof}[Proof of Corollary~\ref{corintro:Salvetti}]
Let $\mathcal B$ be the building of $G(\Gamma)$, and let $\mathcal U^\lambda_\mathcal B(\mathbf X) = \mathcal U(\mathbf X) \leq \Aut(\mathcal B)$ be the universal group with regular cyclic local actions, so that $G(\Gamma)$ is isomorphic to $\mathcal U(\mathbf X) \leq \Aut(\mathcal B)$ by Proposition~\ref{prod:graph-products-are-univ-groups}. Thus  $\lambda$ is the  legal  colouring afforded by Lemma~\ref{lem:legal-colouring}. For each $s \in S$, we may thus view $\lambda_s$ as taking its values in $\mathbf Z \cong X_s$. Let also $L = \mathcal U(\mathbf Y) $ be the universal group with infinite dihedral local actions: for each $s \in S$, the permutation group $Y_s \leq \Sym(\mathbf Z)$ is dihedral and contains the cyclic group $X_s \cong \mathbf Z$ as an index~$2$ subgroup. As remarked in \S\ref{sec:blow-up}, the legal colouring $\lambda$ may be viewed as a $\mathbf Z$-blow-up data, which we denote  by $\mathscr H$. Moreover, the associated graph $|\mathcal B|_\mathscr H$ is isomorphic to $X(\Gamma)^{(1)}$. By definition, the blow-up data $\mathscr H$ is $L$-equivariant. It follows that the $L$-action on $\mathcal B$ preserves the graph  $|\mathcal B|_\mathscr H$. Hence $L \leq \Aut(X(\Gamma)^{(1)})$. 

We have seen in Proposition~\ref{prop:LattEnv} that $L$ is non-discrete as soon as $\Gamma$ is not the complete graph. Conversely, if $\Gamma$ is the complete graph, then $G(\Gamma)$ is free abelian and  $X(\Gamma)$ is the standard cubulation of the Euclidean space, so that $\Aut(X(\Gamma)^{(1)})$ is discrete. This proves (i). 

Since $X(\Gamma)^{(1)}$ is a connected locally finite graph, its automorphism group is a totally disconnected locally compact group. Since $G(\Gamma)$ acts regularly on the vertices of $X(\Gamma)^{(1)}$, it is a cocompact lattice in  $\Aut(X(\Gamma)^{(1)})$. In particular, $G(\Gamma)$ is quasi-isometric to $\Aut(X(\Gamma)^{(1)})$. Hence the assertions (ii), (iii) and (iv) follow from Theorem~\ref{thmintro:2} and Theorem~\ref{thm:Main-latt-env}: indeed, we have seen above that  $L \leq \Aut(X(\Gamma)^{(1)})$, while Theorem~\ref{thm:Main-latt-env} implies that $\Aut(X(\Gamma)^{(1)}) \leq L \rtimes \Aut(\Gamma)$. 
\end{proof}

\bibliographystyle{amsalpha} 
\bibliography{RAAGs}

\end{document}